\documentclass{article}
\usepackage[a4paper]{geometry}
\usepackage{amssymb}
\usepackage{systeme}
\usepackage[english]{babel}
\usepackage{color}
\usepackage{amsmath}
\usepackage{amsthm}
\usepackage{graphicx}
\usepackage{geometry}
\usepackage{comment}
\usepackage[normalem]{ulem}

\usepackage[urlcolor=black]{hyperref}
\hypersetup{
    colorlinks=true, 
    linktoc=all,     
    linkcolor=black,  
    citecolor=black
}
\newcommand\numberthis{\addtocounter{equation}{1}\tag{\theequation}}
\allowdisplaybreaks
\theoremstyle{definition}
\newtheorem{theorem}{Theorem}[section]
\newtheorem{definition}[theorem]{Definition}

\newtheorem{corollary}[theorem]{Corollary}
\newtheorem{lemma}[theorem]{Lemma}
\newtheorem{remark}[theorem]{Remark}
\newtheorem{proposition}[theorem]{Proposition}

\title{Quaternionic K\"{a}hler metrics associated to special K\"{a}hler manifolds with mutually local variations of BPS structures}
\date{\small Department of Mathematics\\
University of Hamburg\\
Bundesstraße 55, D-20146 Hamburg, Germany}
\author{Vicente Cort\'es and Iv\'an Tulli}
\numberwithin{equation}{section}

\begin{document}

\maketitle
\begin{abstract}
    We construct a quaternionic-K\"{a}hler manifold from a conical special K\"{a}hler manifold with a certain type of mutually-local variation of BPS structures. We give global and local explicit formulas for the quaternionic-K\"{a}hler metric, and specify under which conditions it is positive-definite.  Locally, the metric is a deformation of the 1-loop corrected Ferrara-Sabharval metric obtained via the supergravity c-map. The type of quaternionic-K\"{a}hler metrics we obtain are related to work in the physics literature by  S. Alexandrov and S. Banerjee, where they discuss the hypermultiplet moduli space metric of type IIA string theory, with mutually local D-instanton corrections. 
\end{abstract}

\textit{Keywords: quaternionic-K\"{a}hler metrics, variation of BPS structures, HK-QK correspondence, hyperk\"{a}hler metrics, instanton corrections.}\\

\textit{MSC class: 53C26}
\tableofcontents
\section{Introduction}

The supergravity c-map arises in the physics literature in the context of Calabi-Yau compactifications of type IIA and type IIB string theories. Given a fixed Calabi-Yau $3$-fold $X$, the supergravity c-map takes the vector multiplet moduli space $\mathcal{M}^{\text{IIA/B}}_{\text{VM}}(X)$ to the hypermultiplet moduli space $\mathcal{M}^{\text{IIB/A}}_{\text{HM}}(X)$ with its string-tree-level metric $g_{\text{FS}}^0$. The metric $g_{\text{FS}}^0$ is known as the Ferrara-Sabharwal metric \cite{TypeIIgeometry,FSmetric}, and receives both perturbative and non-perturbative quantum corrections. After including the perturbative corrections, one obtains the $1$-loop corrected Ferrara-Sabharwal metric $g_{\text{FS}}^{c}$ \cite{RLSV,1loopcmap}. Mathematically, the 1-loop corrected c-map can be understood as a way of producing a 1-parameter family of quaternionic-K\"{a}hler (QK) manifolds $\{(\overline{N}_c,g_{\text{FS}}^c)\}_{c\in \mathbb{R}}$ from a projective special K\"{a}hler manifold (PSK) \cite{QKPSK}. When one sets the one loop parameter $c\in \mathbb{R}$ to $c=0$, the tree-level metric is recovered.\\ 

On the other hand, there is the much simpler rigid c-map \cite{TypeIIgeometry}, which arises in the context of $4$d $\mathcal{N}=2$ theories compactified on $S^1$.  Mathematically, the rigid c-map can be understood as a way of producing a hyperk\"{a}hler (HK) manifold from an affine special K\"{a}hler (ASK) manifold \cite{SK,ACD}. The HK metric one obtains is sometimes referred to as the ``semi-flat" metric \cite{LCS}.\\

The rigid and the supergravity c-map have been studied in both the mathematics and physics literature, and they can be related via the so-called HK-QK correspondence \cite{HKQK,Conification,QKPSK,WCHKQK}. The HK-QK correspondence takes as input an HK manifold with a rotating $S^1$-action and a certain hyperholomorphic circle bundle; and produces a 1-parameter family of QK manifolds of the same dimension as the HK manifold. By rotating $S^1$-action we mean that it acts by isometries, rotates two of the complex structures, and fixes the remaining complex structure.\\

On the other hand, the QK metric $g^c_{\text{FS}}$ on $\mathcal{M}^{\text{IIA/B}}_{\text{HM}}(X)$ obtained via the 1-loop corrected c-map is expected to receive several non-perturbative quantum corrections in the form of D-instanton and NS$5$-instanton corrections, preserving the QK property. The inclusion of all such corrections has not been fully understood, but a lot of progress has been made in the physics literature via the use of twistor space methods (see for example the extensive reviews \cite{HMreview1,HMreview2} and the references therein). In particular, the inclusion of D-instanton corrections on the type IIA side can be found in \cite{Dinsttwist,WCHKQK}. \\

When considering only the so-called ``mutually-local" D-instantons, the instanton corrected QK metric on $\mathcal{M}^{\text{IIA}}_{\text{HM}}(X)$ has been computed in \cite{HMmetric}. Their computation uses the description of the twistor space of a QK manifold given in \cite{linearpertQK}, together with the description of D-instanton corrections in twistor space given in \cite{Dinsttwist}. Some work is then needed to extract the QK metric from the twistor space data. On the other hand, in \cite{WCHKQK} the twistorial construction of the QK metric with D-instanton corrections was related via QK/HK correspondence to the twistorial construction of the instanton corrected HK metric of the Coulomb branch of  $\mathcal{N}=2$ gauge theories on $\mathbb{R}^3\times S^1$, given in \cite{GMN}. In particular, one should be able to describe the case of mutually-local D-instanton corrections from \cite{HMmetric} in terms of a semi-flat HK geometry corrected by mutually local instanton corrections, together with an appropriate hyperholomorphic line bundle. The description of the instanton corrected HK geometry in \cite{GMN} relies on a set of integral equations, determining the Darboux coordinates of the twistor family of holomorphic symplectic forms of the HK manifold. In the case of mutually local instanton corrections, the integral equations reduce to integral formulas, allowing for an explicit study of this case.\\

The main objective of this work is to give a mathematical treatment of QK metrics involving a mathematical notion of ``mutually local D-instanton corrections". Our methods are 
based on a direct approach to the QK metric avoiding the use of twistor space
and, thus, the technical difficulties arising in the latter framework.  More specifically:

\begin{itemize}
    \item We will start with an integral conical affine special K\"{a}hler (CASK) manifold $(M,g_M,\omega_M,\nabla,\xi,\Lambda)$ (see Definitions \ref{defASKint} and \ref{defCASK}) and a mutually-local variation of BPS structures $(M,\Lambda,Z,\Omega)$ (see Definition \ref{defVarBPS} and \ref{defMutLocBPS}), where $(M,\Lambda,Z)$ is determined by the data of the integral CASK manifold. The notion of variation of BPS structures \cite{VarBPS} is the part of the data that will be used to implement the physical notion of ``D-instanton corrections" to the QK metric. 
    \item Following \cite{GMN}, we will explicitly describe an HK structure  $(N:=T^*M/\Lambda^*,g_N,\omega_1,\omega_2,\omega_3)$ built out of $(M,g_M,\omega_M,\nabla,\xi,\Lambda)$ and $\Omega$, and refer to it as the ``instanton corrected" HK structure. As we mentioned before, the fact that the variation of BPS structures is mutually local circumvents the need of solving the ``GMN integral equations" (since in this case they become integral formulas) and dealing with wall-crossing behavior. Our attitude will be to take the formulas (\ref{holsym}) and (\ref{invKF}) defining the candidate K\"{a}hler forms $\omega_{\alpha}$ for $\alpha=1,2,3$ as an ansatz, and explicitly state under which conditions they define an HK structure without using twistor space arguments. A key notion will be a non-degeneracy condition of a certain tensor field $T$ involving the ASK metric $g_M$ and the BPS indices $\Omega$, that will guarantee the non-degeneracy of $\omega_{\alpha}$ for $\alpha=1,2,3$ (see Definition \ref{defT}). The main result of Section \ref{instcorrsec} is Theorem \ref{theorem1}. This theorem applies to ASK manifolds admitting a central charge (see Definition \ref{centralc}), which includes the CASK case, where there is a canonical choice of central charge (see Proposition \ref{centralchargeCASK}):
    
    \textbf{Theorem 3.13:}  Consider an integral ASK manifold $(M,g_M,\omega_M,\nabla,\Lambda)$ admitting a central charge $Z$, together with a mutually local variation of BPS structures $(M,\Lambda,Z,\Omega)$. Then the triple $(\omega_1,\omega_2,\omega_3)$ of real $2$-forms on $N$ given in (\ref{holsym}) and (\ref{invKF}) defines a pseudo-HK structure on $N$ if and only if the tensor field $T$ on $N$ given in (\ref{non-deg}) is horizontally non-degenerate with respect to the canonical projection $\pi:N\to M$. 

    \item Restricting back to the CASK case, if $J$ is the complex structure on $M$ and $\xi$ the Euler vector field, it was shown in \cite{Conification} that one can lift $J\xi$ to $N$, giving an infinitesimal rotating circle action for the semi-flat HK metric. We will show that the same lift also defines an infinitesimal rotating circle action for the instanton corrected HK structure $(N,g_N,\omega_1,\omega_2,\omega_3)$, see Proposition \ref{rotatingprop}. This, together with the construction of the appropriate hyperholomorphic circle bundle over $N$ in Section \ref{hyperholsec} will allow us to apply the explicit formulas of \cite[Theorem 2]{QKPSK} for the HK-QK correspondence, and hence obtain an ``instanton corrected" QK manifold $(\overline{N},g_{\overline{N}})$. The main result of Section \ref{hyperholsec} is Theorem \ref{theorem2}:
    
    \textbf{Theorem 4.10:} Let  $(M,g_M,\omega_M,\nabla,\xi,\Lambda)$ be a connected integral CASK manifold together with a mutually local $(M,\Lambda,Z,\Omega)$, where $Z$ is the canonical central charge of a CASK manifold. Assume that $T$ is horizontally non-degenerate and that the flow of $\xi$ generates a free-action of the monoid $\mathbb{R}_{\geq 1}$. Furthermore, let $(N,g_N,\omega_1,\omega_2,\omega_3)$ be the associated pseudo-HK manifold; $(P\to N,\eta)$ the associated hyperholomorphic circle bundle; $N'\subset N$ the (non-empty) open subset defined in (\ref{N'}) and $\theta_i^P\in \Omega^1(P)$, $g_P\in \text{Sym}^2(P)$, $X_1^P\in \Gamma(TP)$, $f\in C^{\infty}(P)$ the associated objects defined in Section \ref{HK-QK}. If $\overline{N}\subset P|_{N'}$ is any submanifold transversal to $X_1^P$, then 

\begin{equation}\label{QKmetricglobalintro}
    g_{\overline{N}}:=-\frac{1}{f}\Big(g_P - \frac{2}{f}\sum_{i=0}^3 (\theta_i^P)^2\Big)\Big|_{\overline{N}}
\end{equation}
is a pseudo-QK metric on $\overline{N}$. Furthermore, if $N'_{+}\subset N'$ is as in Lemma \ref{Nrestriction} and $\overline{N}$ is picked to also satisfy $\overline{N}_+:=\overline{N}\cap N_+'\neq \emptyset$, then $g_{\overline{N}}$ is positive definite on $\overline{N}_+$.
    \item In Section \ref{QKdefFS} we apply Theorem \ref{theorem2} to the case of a CASK domain, and describe the resulting $(\overline{N},g_{\overline{N}})$ in coordinates that realize $g_{\overline{N}}$ as a deformation of the $1$-loop corrected Ferrara-Sabharval metric $g_{\text{FS}}^c$, see Theorem \ref{propcoordQK}. We furthermore show in Proposition \ref{signatureQK} that the subset $\overline{N}_{+}$ where $g_{\overline{N}}$ is positive definite in never empty (for our choice of $\overline{N})$, and say something about the fate of the Peccei-Quinn symmetries after instanton corrections in Corollary \ref{PQsymm}. 
    \item Finally, in Section \ref{example} we give a simple example of our constructions. In this example, we consider a CASK domain whose associated PSK manifold is the complex hyperbolic space $\mathbb{C}H^n$ with the  Bergman metric. If $\rho$ denotes the usual dilaton coordinate, our example then gives a deformation $g_{\overline{N}}$ of $g_{FS}^c$ in a neighborhood of $\rho=\infty$, where both $g_{\overline{N}}$ and $g_{FS}^c$ are defined and positive definite (see Corollary \ref{comhyperdef}). In particular, when $n=0$, we obtain an instanton deformation of the universal hypermultiplet, near its infinite end of finite volume (see Corollary \ref{finitevolend}).
\end{itemize}

Regarding our results and the related work \cite{HMmetric} in the physics literature:
\begin{itemize}
    \item  The argument we follow is similar in spirit to the one in \cite{WCHKQK}, and gives a formal mathematical proof that the final metric is QK. Furthermore, we obtain a global formula for the QK metric (\ref{QKmetricglobalintro}), while in \cite{HMmetric} only a local formula is found. 
    \item In the local case of a CASK domain, the coordinate expression of $g_{\overline{N}}$ in Theorem \ref{propcoordQK} should be compared with the corresponding expression in \cite{HMmetric}. In our case, $g_{\overline{N}}$ turns out to be completely explicit, avoiding the use of the implicitly defined $\mathcal{R}$-parameter in \cite{HMmetric}. We do however express everything in terms of ``classical coordinates", whereas in \cite{HMmetric} a ``quantum corrected" dilaton coordinate is used. See also Section \ref{commentssec} for some further comments on this. 
    \item On the other hand, some effort has been made in indicating which expressions are due to ``instanton corrections". This allows for a direct comparison with the usual 1-loop corrected Ferrara-Sabharval metric $g^c_{\text{FS}}$, and realize the QK metric we obtain in the case of a CASK domain as  deformations of $g^c_{\text{FS}}$.

\end{itemize}

\textbf{Acknowledgements:} This work was supported by the Deutsche Forschungsgemeinschaft (German Research Foundation) under Germany’s Excellence Strategy -- EXC 2121 ``Quantum Universe'' -- 390833306. The idea for this work originated from discussions within our Swampland seminar, which is part of the Cluster of Excellence Quantum Universe. The authors would like to thank Murad Alim, J\"{o}rg Teschner, Timo Weigand and Alexander Westphal for the aforementioned discussions. Furthermore, the authors would like to thank Danu Thung for comments on the draft.

\section{Review of special K\"{a}hler manifolds and the rigid c-map}

In this section we collect a few facts about affine special K\"{a}hler (ASK) manifolds, the rigid c-map, conical affine special K\"{a}hler (CASK) manifolds and projective special K\"{a}hler (PSK) manifolds. Our main references are \cite{SK,ACD,CM}.\\

In the final subsection \ref{GMNsemiflat} we give a description of the semi-flat hyperk\"{a}hler metric in terms of the notion of central charges. This description will be convenient for the constructions of the later sections. 

\begin{subsection}{Affine special K\"{a}hler manifolds}

\begin{definition}
an affine special K\"{a}hler (ASK) manifold is a tuple $(M,g,\omega,\nabla)$, where:

\begin{itemize}
    \item $(M,g,\omega)$ is a pseudo-K\"{a}hler manifold. We denote by $J$ the corresponding complex structure determined by the relation $g(J-,-)=\omega(-,-)$.
    \item $\nabla$ is a flat, torsion-free connection on $M$.
    \item $\nabla \omega=0$ and $d_{\nabla}J=0$, where $d_{\nabla}:  \Omega^n_M(TM)\to \Omega^{n+1}_M(TM)$ is the extension of $\nabla: \Omega^0_M(TM)\to \Omega^1_M(TM)$  to higher degree forms, and we think of $J$ as an element of $\Omega^1_M(TM)$. 
\end{itemize}
\end{definition}

The fact that $\nabla$ is flat and torsion-free, together with $\nabla \omega=0$, implies that we can find $\nabla$-flat Darboux coordinates $(x^i,y_i)$ for $\omega$, i.e:

\begin{equation}\label{SAC}
    \omega=dx^i\wedge dy_i \;\;\;\; \text{with} \;\;\;\; \nabla dx^i=0, \;\; \nabla dy_i=0\,.
\end{equation}
\begin{definition}
Coordinates $(x^i,y_i)$ in an ASK manifold satisfying (\ref{SAC}) are called affine special coordinates. 
\end{definition}

On the other hand, the torsion free condition can be written as $d_{\nabla}(\text{Id})=0$ where $\text{Id} \in \Omega^1_M(TM)$ is the identity map. Together with  $d_{\nabla}J=0$, this implies that the projection $\pi^{1,0}=\frac{1}{2}(\text{Id}-iJ)$ satisfies 

\begin{equation}
    d_{\nabla}\pi^{1,0}=d_{\nabla}\Big(\frac{1}{2}(\text{Id}-iJ)\Big)=0\,.
\end{equation}
By the Poincar\'e lemma, we can locally find a complex vector field $\xi^{1,0}$ on $M$ such that
\begin{equation}
    \nabla \xi^{1,0}=\pi^{1,0}\,.
\end{equation}
Consider a local flat Darboux frame $(\widetilde{\gamma}_i,\gamma^i)$ with respect to $\omega$. We can then write

\begin{equation}\label{HSC}
        \xi^{1,0}=\frac{1}{2}(z^i\widetilde{\gamma}_i-w_i\gamma^i)\,,
    \end{equation}
and obtain local complex valued functions $z^i$ and $w_i$ on $M$. The flatness of $(\widetilde{\gamma}_i,\gamma^i)$ together with $\nabla \xi^{1,0}=\pi^{1,0}$ imply that $z^i$ and $w_i$ must be holomorphic. Furthermore, if we define $x^i:=\text{Re}(z^i)$ and $y_i:=-\text{Re}(w_i)$, then using the fact that $2\text{Re}(\nabla \xi^{1,0})=2\text{Re}(\pi^{1,0})=\text{Id}$, we find that $(x^i,y_i)$ is a local coordinate system and that $\partial_{x^i}=\widetilde{\gamma}_i$, $\partial_{y_i}=\gamma^i$. In particular $(x^i,y_i)$ is an affine special coordinate system and $(z^i)$, $(w_i)$ are systems of holomorphic coordinates.

\begin{definition}
A pair of holomorphic coordinate systems $(z^i)$ and $(w_i)$ satisfying (\ref{HSC}) is called a conjugate system of special holomorphic coordinates.
\end{definition}
 With respect to the special holomorphic coordinates $z^i$ (or $w_i$), one can locally describe the ASK geometry in terms of a holomorphic function $\mathfrak{F}(z^i)$ (usually known as the holomorphic prepotential) satisfying 

\begin{equation}
    w_i=\frac{\partial{\mathfrak{F}}}{\partial z^i}\,.
\end{equation}
The function $\mathfrak{F}$ locally describes the ASK geometry in the sense that 

\begin{equation}
    \omega=\frac{i}{2}\text{Im}(\tau_{ij})dz^i\wedge d\overline{z}^j\;\;\;\;\; \text{with} \;\;\;\;\; \tau_{ij}:=\frac{\partial^2 \mathfrak{F}}{\partial z^i \partial z^j}\,.
\end{equation}
We will also make frequent use of the following additional data:
\begin{definition}\label{defASKint} An integral ASK manifold is a tuple $(M,g,\omega,\nabla,\Lambda)$ such that:
\begin{itemize}
    \item $(M,g,\omega,\nabla)$ is an ASK manifold.
    \item  $\Lambda\subset TM$ is a bundle $\Lambda \to M$ of $\nabla$-flat lattices such that $\Lambda \otimes_{\mathbb{Z}} \mathbb{R}=TM$.
    \item Around any $p\in M$, $\Lambda\to M$ admits a Darboux frame with respect to $\omega$.
\end{itemize} 
\end{definition}

\end{subsection}
\subsection{The rigid c-map}

Given an ASK manifold $(M,g,\omega,\nabla)$ of signature $(n,m)$, the rigid c-map associates to it a  pseudo-hyperk\"{a}hler (HK) structure of signature $(2n,2m)$ on the cotangent bundle  $(T^*M,g^{\text{sf}},I_1,I_2,I_3)$. We use the superscript $^{\text{sf}}$ since the metric is ``semi-flat" in the sense that it restricts to a flat metric on the fibers. \\

The HK structure is defined as follows. Let $\pi:T^*M\to M$ be the canonical projection, then the connection $\nabla$ allows us to do a splitting

\begin{equation}
    T(T^*M)=T^h(T^*M)\oplus T^v(T^*M)\,,
\end{equation}
where $T^h(T^*M)\cong \pi^*TM$ and $T^v(T^*M)=\text{Ker}(d\pi)\cong \pi^*T^*M$. With respect to this splitting we have

\begin{equation}\label{c-map}
    g^{\text{sf}}:=g\oplus g^{-1}, \;\;\;\;\; I_3:=J\oplus J^*, \;\;\;\;\; I_1(v,w):= -\omega^{-1}(w) + \omega(v), \;\;\;\;\; I_2:=I_3I_1 \,,
\end{equation}
where in the definition of $I_1$ we think of $\omega$ as a map $\omega: TM\to T^*M$. We remark that the fact that this defines an actual HK structure on $T^*M$ really uses all the special K\"{a}hler conditions (see e.g.\  \cite[Section 2]{SK} or  \cite[Theorem 11]{ACD}). The numbering of the $I_{\alpha}$ is set to match our later conventions.\\

If $(M,g,\omega,\nabla,\Lambda)$ is an integral ASK manifold,  then the dual lattice $\Lambda^*\subset T^*M$ is Lagrangian with respect to the canonical holomorphic symplectic form on $T^*M$ (which coincides with $\omega_1^{\text{sf}}+i\omega_2^{\text{sf}}=g^{\text{sf}}(I_1-,-)+ig^{\text{sf}}(I_2-,-)$), and we have:\\

\begin{theorem}\cite[Theorem 3.1]{C} and  \cite[Theorem 3.4, Theorem 3.8]{SK}: Let $(M,g,\omega,\nabla,\Lambda)$ be an integral ASK manifold. Let $N:=T^*M/\Lambda^*$ and consider the canonical projection  $\pi:N\to M$. Then $\pi:N\to M$ has the structure of an integrable system, and $N$ has a canonical pseudo-HK structure induced from $(T^*M,g^{\text{sf}},I_1,I_2,I_3)$. 
\end{theorem}

\begin{subsection}{CASK and PSK manifolds}\label{CASKPSKdef}

\begin{definition}\label{defCASK}
A conical affine special K\"{a}hler (CASK) manifold is a tuple $(M,g,\omega,\nabla,\xi)$ where:
\begin{itemize}
    \item $(M,g,\omega,\nabla)$ is an ASK manifold. We denote the complex structure by $J$.
    \item $\nabla \xi= D\xi=\text{Id}$, where $D$ denotes the Levi-Civita connection.
    \item $g$ is positive definite on $\mathcal{D}=\text{span}\{\xi,J\xi\}$ and negative definite on $\mathcal{D}^{\perp}$.
\end{itemize}
Furthermore, an integral CASK manifold $(M,g,\omega,\nabla,\xi,\Lambda)$ is just an integral ASK manifold that is also CASK (see Definition \ref{defASKint}).
\end{definition}

The vector fields $\xi$ and $J\xi$  satisfy the following identities with the Lie derivative: $\mathcal{L}_{\xi}J=\mathcal{L}_{J\xi}J=0$, $\mathcal{L}_{\xi}g=2g$, $\mathcal{L}_{J\xi}g=0$, and $\mathcal{L}_{\xi}J\xi=0$ \cite{CM}. In other words, $\xi$ is holomorphic and homothetic, $J\xi$ is holomorphic and Killing, and they commute. \\

In some cases we will assume that the holomorphic Killing vector $J\xi$ generates a free $S^1$-action on $M$, and that the holomorphic homothetic vector $\xi$ generates a free $\mathbb{R}_{>0}$-action on $M$. Under such a condition, the function $\mu:=r^2/2$ where $r:=\sqrt{g(\xi,\xi)}$, is a moment map for the $S^1$-action, and if we define

\begin{equation}
    S:=\{p \in M \;\;\; |\;\;\; g_p(\xi,\xi)=1\}\,,
\end{equation}
then $-g|_{S}$ induces a positive definite K\"{a}hler metric $g_{\overline{M}}$ on $\overline{M}:=S/S^1_{J\xi}=M//S^1_{J\xi}$ (i.e. the K\"{a}hler quotient). The relations between $g$, $g|_S$ and $g_{\overline{M}}$ can be summarized as follows \cite{QKPSK}: let $\pi_{S}:M\to M/\mathbb{R}_{>0}=S$ and $\pi_{\overline{M}}:S\to S/S^1=\overline{M}$ be the projections, and let
    \begin{equation}\label{contform}
        \widetilde{\eta}:=\frac{1}{r^2}g(J\xi,-)=d^c \log(r)=i(\overline{\partial}-\partial)\log(r)\,.
    \end{equation}
    We then have
    \begin{equation}\label{metriccone}
        g=dr^2+r^2\pi_S^*g|_S \;\;\;\;\;\;\; g|_S=\widetilde{\eta}^2|_S - \pi^*_{\overline{M}}g_{\overline{M}}\,.
    \end{equation}

\begin{definition}
A projective special K\"{a}hler (PSK) manifold is a K\"{a}hler manifold $(\overline{M},g_{\overline{M}},\omega_{\overline{M}})$ obtained from a CASK manifold by the K\"{a}hler quotient from before. 
\end{definition}

Given a CASK manifold $(M,g,\omega,\nabla,\xi)$ we can locally  find a special affine coordinate system $(x^i,y_i)$ such that \cite{CM}:
    \begin{equation}\label{CSAC}
        \xi=x^i\partial_{x^i}+y_i\partial_{y_i}\,.
    \end{equation}
\begin{definition}
A special affine coordinate system $(x^i,y_i)$ satisfying (\ref{CSAC}) is called a conical special affine coordinate system.
\end{definition}
 On the other hand, the global complex vector field 
    \begin{equation}\label{canonical10Euler}
        \xi^{1,0}:=\frac{1}{2}(\xi -i J \xi)
    \end{equation}
    satisfies $\nabla \xi^{1,0}=\pi^{1,0}$, where $\pi^{1,0}=\frac{1}{2}(\text{Id}-iJ)$ is the projection $ TM\otimes \mathbb{C}\to T^{1,0}M$ \cite{CM}. By picking a local flat Darboux frame $(\widetilde{\gamma_i},\gamma^i)$ for $\omega$ we can write
    
    \begin{equation}\label{CSHC}
        \xi^{1,0}=\frac{1}{2}(z^i\widetilde{\gamma}_i-w_i\gamma^i)
    \end{equation}
    and obtain holomorphic coordinate systems $(z^i)$ and $(w_i)$. It is easy to check that $(x^i:=\text{Re}(z^i),y_i:=-\text{Re}(w_i))$ give a system of conical special affine coordinates. 
    \begin{definition}\label{conicalholspecialcoords}
    When $(z^i)$ and $(w_i)$ are defined by (\ref{CSHC}) using the globally defined $\xi^{1,0}=\frac{1}{2}(\xi-iJ\xi)$ we will call $(z^i)$ and $(w_i)$ a conjugate system of conical special holomorphic coordinates. \end{definition}

     If $(z^i)$ and $(w_i)$ are a conjugate system of conical special holomorphic coordinates, then they are homogeneous of degree $1$ with respect to the (local) $\mathbb{C}^{\times}$-action generated by $\{\xi,J\xi\}$ (i.e. $\mathcal{L}_{\xi}z^j=z^j$, $\mathcal{L}_{J\xi}z^j=iz^j$, $\mathcal{L}_{\xi}w^j=w^j$, $\mathcal{L}_{J\xi}w^j=iw^j$). In particular, this implies that $\mathcal{L}_{\xi^{1,0}}z^i=z^i$, so that
     
     \begin{equation}\label{Ctimesgen}
         \xi^{1,0}=z^i\partial_{z^i}\,.
     \end{equation}
If we define $\tau_{ij}$ by the relation $dw_i=\tau_{ij}dz^j$, then a consequence of (\ref{CSHC}) and (\ref{Ctimesgen}) is that $w_i=\tau_{ij}z^j$. This furthermore implies that $\mathfrak{F}(z^i):=\frac{1}{2}\tau_{ij}z^iz^j$ is a holomorphic prepotential for the CASK geometry. \\
    
    Finally, we remark that for a CASK manifold, the map $r:=\sqrt{g(\xi,\xi)}$ gives a global K\"{a}hler potential for $\omega$:
    \begin{equation}\label{globalKP}
        \omega=\frac{i}{2}\partial \overline{\partial}r^2\,.
    \end{equation}
    In conical holomorphic special coordinates $(z^i)$, this follows from $\frac{\partial\tau_{ij}}{\partial{z}^k}z^j=0$ (which in turn follows from the CASK relation $w_i=\tau_{ij}z^j$) together with $r^2=\text{Im}(\tau_{ij})z^i\overline{z}^j$.

\end{subsection}

\begin{subsection}{Central charges and the semi-flat metric}\label{GMNsemiflat}

Let $(M,g,\omega,\nabla,\Lambda)$ be an integral ASK manifold. Here we present another description of the associated semi-flat HK metric $(N=T^*M/\Lambda^*,g^{\text{sf}},\omega_1,\omega_2,\omega_3)$. This description is closer to the language of $4d$ $\mathcal{N}=2$ SUSY theories, and will be more useful when we include the ``instanton corrections" in the form of variations of BPS structures \cite{GMN,NewCHK}. 

\begin{definition}
Given an integral ASK manifold $(M,g,\omega,\nabla,\Lambda)$, $\omega|_{\Lambda \times \Lambda}$ defines an integral skew pairing on $\Lambda$ that we will denote by

\begin{equation}
    \langle - , - \rangle := \omega|_{\Lambda\times \Lambda}: \Lambda \times \Lambda \to \mathbb{Z}\,.
\end{equation}

By the definition of an integral ASK manifold, $\langle - , - \rangle$ admits local Darboux frames $(\widetilde{\gamma}_i,\gamma^i)$. Our convention will be that $\langle \widetilde{\gamma}_i,\gamma^j \rangle=\delta_i^j$.
\end{definition}

Let $(\widetilde{\gamma}_i,\gamma^i)$ be a Darboux frame of $\Lambda$ over $U\subset M$. By possibly restricting $U$ we can find a local complex vector field $\xi^{1,0}$ on $U$ such that $\nabla \xi^{1,0}=\pi^{1,0}$, and consider the corresponding system of conjugate special holomorphic coordinates $(z^i)$ and $(w_i)$ determined by

    \begin{equation}\label{intspecial}
        \xi^{1,0}=\frac{1}{2}\Big( z^i\widetilde{\gamma}_i- w_i\gamma^i\Big)\,.
    \end{equation} 
\begin{definition}
A conjugate system of holomorphic special coordinates $(z^i)$, $(w_i)$ defined by (\ref{intspecial}) with respect to a local Darboux frame of $\Lambda$ will be called a conjugate system of integral holomorphic special coordinates. 
\end{definition}

Two such overlapping systems of conjugate holomorphic special coordinates are related by a (constant) transformation in $\mathbb{C}^{2n}\rtimes Sp(2n,\mathbb{Z})$ (the $\mathbb{C}^{2n}$ factor is there because we can always shift $\xi^{1,0}$ by a complex parallel vector field when solving $\nabla \xi^{1,0}=\pi^{1,0}$). 

\begin{proposition} \label{centralchargeprop} The following are equivalent:
\begin{itemize}
\item $(M,g,\omega,\nabla,\Lambda)$ admits a covering by conjugate systems of integral holomorphic special coordinates $\{(U_{\alpha},(z^i_{\alpha}),(w_{i,\alpha}))\}_{\alpha}$ related on overlaps by a (constant) transformation in $Sp(2n,\mathbb{Z})$. 
    \item There is a holomorphic section $Z$ of $\Lambda^*\otimes \mathbb{C}\to M$ such that for any local Darboux frame $(\widetilde{\gamma}_i,\gamma^i)$ of $\Lambda$, the holomorphic functions $(Z_{\gamma^i})$ and $(Z_{\widetilde{\gamma}_i})$ give a conjugate system of integral holomorphic special coordinates. 
\end{itemize}

\end{proposition}
\begin{remark} In the above proposition $\Lambda^*\to M$ denotes the dual of $\Lambda \to M$. Furthermore, for $\gamma$ a section $\Lambda$ over $U\subset M$, $Z_{\gamma}:U\to \mathbb{C}$ denotes the corresponding holomorphic function obtained by contracting with $\gamma$.

\end{remark}
\begin{proof}
Consider $(U_{\alpha},\{z^i_{\alpha}\},\{w_{i,\alpha}\})$ for some $\alpha$, and the corresponding Darboux frame $(\widetilde{\gamma}_{i,\alpha},\gamma^i_{\alpha})$ of $\Lambda$ over $U_{\alpha}\subset M$. We locally define $Z:U_{\alpha} \to \Lambda^*\otimes \mathbb{C}$ as follows: given a section $\gamma$ of $\Lambda|_{U_{\alpha}}$ 
we can write  $\gamma=n^i\widetilde{\gamma}_{i,\alpha}+n_i\gamma^i_{\alpha}$, and then define
    \begin{equation}
      Z_{\gamma}=n^iw_{i,\alpha}+n_iz^i_{\alpha}\,.
    \end{equation}
    The fact that the conjugate systems of holomorphic special coordinates are related by a transformation in $Sp(2n,\mathbb{Z})$ then implies that the local definitions of $Z$ glue together into a holomorphic section $Z:M \to \Lambda^*\otimes \mathbb{C}$. Furthermore, given any other Darboux frame $(\widetilde{\gamma_i},\gamma^i)$ of $\Lambda$ over $U$, it is easy to check that if $U\cap U_{\alpha}\neq \emptyset$ then 
    \begin{equation}
         \frac{1}{2}\Big( Z_{\gamma^i}\widetilde{\gamma}_i - Z_{\widetilde{\gamma}_i}\gamma^i\Big)=\frac{1}{2}\Big(Z_{\gamma^i_{\alpha}}\widetilde{\gamma}_{i,\alpha}-Z_{\widetilde{\gamma}_{i,\alpha}}\gamma^i_{\alpha}\Big)=\xi^{1,0}\,,
    \end{equation}
    so $\{Z_{\gamma^i}\}$ and $\{Z_{\widetilde{\gamma}_i}\}$ give a conjugate system of integral holomorphic special coordinates. \\
    
    On the other hand, it is clear that if $Z:M \to \Lambda^*\otimes \mathbb{C}$ exists, then the required cover exists. 
\end{proof}
For future reference we note that the section $Z$ of $\Lambda^*\otimes \mathbb{C}$ has the following expansion in the dual Darboux frame  
\begin{equation} \label{Z:eq} Z = Z_{\gamma^i}(\gamma^i)^* + Z_{\tilde{\gamma}_i}(\widetilde{\gamma}_i)^*.\end{equation}
It is related to $\xi^{1,0}$ via $\omega( \xi^{1,0},-)|_{\Lambda} = \frac12 Z$.
\begin{definition}\label{centralc} We will say that an integral ASK manifold $(M,g,\omega,\nabla,\Lambda)$ admits a central charge homomorphism if it has a holomorphic section $Z:M \to \Lambda^*\otimes \mathbb{C}$ satisfying the condition of Proposition \ref{centralchargeprop}.
\end{definition}

\begin{proposition}\label{centralchargeCASK}
An integral CASK manifold $(M,g,\omega,\nabla,\xi,\Lambda)$ has a canonical
central charge homomorphism $Z:M\to \Lambda^*\otimes \mathbb{C}$,
which is unique up to the action of the group of global sections of $Sp(\Lambda )$. If $(\widetilde{\gamma}_i,\gamma^i)$ is a Darboux frame of $\Lambda$, then $(Z_{\gamma^i})$ and $(Z_{\widetilde{\gamma}_i})$ gives a conjugate system of integral, conical, holomorphic special coordinates. 
\end{proposition}

\begin{proof}
In the CASK case we have a global and canonical complex vector field $\xi^{1,0}:=\frac{1}{2}(\xi-iJ\xi)$ satisfying $\nabla \xi^{1,0}=\pi^{1,0}$. Any two conjugate systems of integral holomorphic coordinates defined by this $\xi^{1,0}$ are related by a transformation in $Sp(2n,\mathbb{Z})$. This allows us to produce the required cover of Proposition \ref{centralchargeprop} and hence a canonical central charge homomorphism. The fact that  $(Z_{\gamma^i})$ and $(Z_{\widetilde{\gamma}_i})$ are conical is clear from the fact that we are using $\frac{1}{2}(\xi-iJ\xi)$ to define $Z$.
\end{proof} 

We are now ready to give the description of $g^{\text{sf}}$ that will be useful for the following sections.

\begin{proposition} Let $(M,g,\omega,\nabla,\Lambda)$ be an integral ASK manifold admitting a central charge $Z:M\to \Lambda^*\otimes \mathbb{C}$, and let $(N=T^*M/\Lambda^*,g^{\text{sf}},I_1,I_2,I_3)$ be the associated semi-flat HK manifold. Furthermore, consider the torus bundle $\mathcal{N}\to M$ defined by

\begin{equation}\label{charbundle}
    \mathcal{N}_u:=\{\theta : \Lambda_u \to \mathbb{R}/2\pi \mathbb{Z} \;\; | \;\; \theta_{\gamma+ \gamma'}= \theta_{\gamma}+\theta_{\gamma'}\}\,.
\end{equation}
Then $\mathcal{N}\to M$ is canonically isomorphic to $N \to M$. Furthermore, by using the induced\footnote{We define the induced pairing by the property that $\gamma \mapsto \langle \gamma,-\rangle$ maps the pairing in $\Lambda$ to the pairing in $\Lambda^*$. With this definition the dual of a Darboux basis is a Darboux basis. We will also keep using the notation $\langle - , - \rangle$ for the $\mathbb{C}$-bilinear extension of the pairing to $\Lambda^*\otimes \mathbb{C}$} pairing $\langle - , - \rangle$ on $\Lambda^*$ and the evaluation map $\theta: \mathcal{N} \to \Lambda^*\otimes \mathbb{R}/2\pi \mathbb{Z}$, we can write the K\"{a}hler forms $\omega_{\alpha}^{\text{sf}}:=g^{\text{sf}}(I_{\alpha}-,-)$ as follows (see \cite{NewCHK}):
\begin{equation}
    \begin{split}
    \omega_1^{\text{sf}}+i\omega_2^{\text{sf}}&=-\frac{1}{2\pi}\langle dZ\wedge d\theta \rangle \\
    \omega_3^{\text{sf}}&=\frac{1}{4}\langle dZ\wedge d\overline{Z}\rangle - \frac{1}{8\pi^2}\langle d\theta \wedge d\theta \rangle\,. 
    \end{split}
\end{equation}

\end{proposition}

\begin{proof}
Consider the natural projection $p:\mathbb{R}\to \mathbb{R}/2\pi \mathbb{Z}$, then the bundle isomorphism $\mathcal{N}\cong N$ is given by 
    \begin{equation}
        [\alpha] \in N=T^*M/\Lambda^* \to  p\circ (2\pi\alpha|_{\Lambda}) \in \mathcal{N}\,.
    \end{equation}
    
To check the claim on the K\"{a}hler forms, it is enough to show that in local coordinates we recover the usual expressions. Fix a local Darboux frame $(\widetilde{\gamma}_i,\gamma^i)$ of $\Lambda$ over $U$, and consider the local coordinates $(Z_{\gamma^i},\theta_{\widetilde{\gamma}_i},\theta_{\gamma^i})$ on $\mathcal{N}\cong N$, where $\theta_{\gamma}:\mathcal{N}|_U\to \mathbb{R}/2\pi \mathbb{Z}$ is the contraction of $\theta$ with $\gamma$. We can then write

\begin{equation}
    -\frac{1}{2\pi}\langle dZ\wedge d\theta \rangle=-\frac{1}{2\pi}(dZ_{\widetilde{\gamma}_i}\wedge d\theta_{\gamma^i} - dZ_{\gamma^i}\wedge d\theta_{\widetilde{\gamma}_i})=\frac{1}{2\pi}dZ_{\gamma^i}\wedge (d\theta_{\widetilde{\gamma}_i}-\tau_{ij}d\theta_{\gamma^j})
\end{equation}
where we used $dZ_{\widetilde{\gamma}_i}=\tau_{ij}dZ_{\gamma^j}$. Similarly, using that $\tau_{ij}$ must be symmetric (since $(Z_{\widetilde{\gamma}_i})$ and $(Z_{\gamma^i})$ are a conjugate system of holomorphic special coordinates), we find

\begin{equation}
    \begin{split}
        \frac{1}{4}\langle dZ\wedge d\overline{Z}\rangle - \frac{1}{8\pi^2}\langle d\theta \wedge d\theta \rangle&=\frac{1}{4}(dZ_{\widetilde{\gamma}_i}\wedge d\overline{Z}_{\gamma^i}-dZ_{\gamma^i}\wedge d\overline{Z}_{\widetilde{\gamma}_i})-\frac{1}{8\pi^2}(d\theta_{\widetilde{\gamma}_i}\wedge d\theta_{\gamma^i}- d\theta_{\gamma^i}\wedge d\theta_{\widetilde{\gamma}_i})\\
        &=\frac{i}{2}\text{Im}(\tau_{ij})dZ_{\gamma^i}\wedge dZ_{\gamma^j}-\frac{1}{4\pi^2}d\theta_{\widetilde{\gamma}_i}\wedge d\theta_{\gamma^i}\\
        &=\frac{i}{2}\text{Im}(\tau_{ij})dZ_{\gamma^i}\wedge d\overline{Z}_{\gamma^j} + \frac{i}{8\pi^2}\big(\text{Im}(\tau)\big)^{ij}(d\theta_{\widetilde{\gamma}_i}-\tau_{ik}d\theta_{\gamma^k})\wedge(d\theta_{\widetilde{\gamma}_j}-\overline{\tau_{jl}}d\theta_{\gamma^l})\,.
    \end{split}
\end{equation}

Under our identification, and up to an overall normalization of the metric, these expressions reproduce the usual formulas of the semi-flat HK metric (see for example \cite[Proposition 3]{QKPSK}, while keeping in mind the different sign conventions for special coordinates).
\end{proof}
Note that $\frac{1}{4}\langle dZ\wedge d\overline{Z}\rangle = \frac{i}{2}\partial \bar{\partial}r^2$ is precisely the K\"ahler form of the affine special K\"ahler manifold, where $r^2 =  -\frac{i}{2}\langle Z,\bar{Z}\rangle$. In the case of a conical affine special K\"ahler manifold, we can write this K\"ahler potential as 
$r^2=g(\xi,\xi)$, cf.\ (\ref{Z:eq}).

\end{subsection}
\section{Instanton corrected HK structures}
\label{instcorrsec}

Consider an integral ASK manifold $(M,g_M,\omega_M,\nabla,\Lambda)$. We  seek to include the data of a variation of BPS structures over $M$, to get an ``instanton corrected" pseudo-hyperk\"{a}hler metric on $N=T^*M/\Lambda^*$, following \cite{GMN,NewCHK}.\\

On the other hand, it is known that the semi-flat HK metric on $N$ coming from the CASK manifold $M$ has an infinitesimal rotating circle action \cite{Conification,QKPSK}. We wish to show that this infinitesimal rotating action survives the instanton corrections coming from a mutually local variation of BPS structures, with the aim of applying the HK-QK correspondence in the next sections. \\

We start by quickly reviewing in Section \ref{varsec} the inclusion of instanton corrections according to the work in the physics literature of \cite{GMN}. Section \ref{varsec} is only meant as a motivation for the formulas appearing in Section \ref{mutsec}. In Section \ref{mutsec} we restrict to the simpler case of mutually local corrections, and prove under which conditions we obtain an ``instanton corrected" HK structure (see Theorem \ref{theorem1}). The proof is direct and explicit, avoiding the use of twistor space methods. 

\begin{subsection}{Variations of BPS structures and instanton corrected HK metrics}
\label{varsec}
To explain what we will mean by an instanton correction of the semi-flat HK structure, we will require the notion of variations of BPS structures \cite{VarBPS}.

\begin{definition} \label{defVarBPS}
A variation of (integral) BPS structures over a complex manifold $M$ is a tuple $(M,\Lambda,Z, \Omega)$, where:

\begin{itemize}
    \item $\Lambda \to M$ is a local system of lattices $\Lambda_p\cong \mathbb{Z}^r$ with a covariantly constant,  skew, integer-valued pairing $\langle - ,- \rangle$.
    \item $Z$ is a holomorphic section of $\Lambda^*\otimes \mathbb{C}\to M$.
    \item $\Omega:\Lambda\to \mathbb{Z}$ is a function (of sets) satisfying $\Omega(\gamma)=\Omega(-\gamma)$ and the Kontsevich-Soibelman wall-crossing formula \cite{KS,VarBPS}. 
    \end{itemize}
    
    \begin{remark}\label{WKremark}
    We will not need to fully state the KS wall-crossing formula, but will only mention some consequences that this condition has on $\Omega$. Consider the real codimension $1$ subset $\mathcal{W}\subset M$ defined by
        
        \begin{equation}
            \mathcal{W}:=\{p\in M \;\; | \;\; \exists \gamma, \gamma' \in \text{Supp}(\Omega)\cap \Lambda_p, \;\; \langle \gamma, \gamma' \rangle\neq 0, \;\; Z_{\gamma}/Z_{\gamma'}\in \mathbb{R}_{>0} \}\,,
        \end{equation}
        where $\text{Supp}(\Omega):=\{\gamma \in \Lambda \;\; | \;\; \Omega(\gamma)\neq 0\}$.
        The fact that $\Omega$ satisfies the wall-crossing formula implies that for a local section $\gamma$ of $\Lambda$, $\Omega(\gamma)$ is locally constant on $M\backslash \mathcal{W}$. Furthermore, the discontinuity at the ``wall" $\mathcal{W}$ is completely determined by the wall-crossing formula, and $\Omega$ is monodromy invariant (i.e. if $\gamma_p$ has monodromy $A\cdot \gamma_p$ for $A\in Sp(\Lambda_p, \langle - , -\rangle)$ around a loop, then $\Omega(\gamma_p)=\Omega(A\cdot \gamma_p)$).
    \end{remark}
    
    The tuple $(M,\Lambda,Z, \Omega)$ should furthermore satisfy the following two properties:
    \begin{itemize}
    \item Support property: given a compact set $K\subset M$ and a choice of covariantly constant norm $|\cdot |$ on $\Lambda|_K \otimes_{\mathbb{Z}}\mathbb{R} $, there should be a constant $C>0$ such that for any $\gamma \in \Lambda|_K\cap \text{Supp}(\Omega)$
    \begin{equation} \label{supportproperty}
        |Z_{\gamma}|>C|\gamma|\,.
    \end{equation}
    
    \item Convergence property: for any $R>0$ the series
    \begin{equation}\label{convergenceproperty}
        \sum_{\gamma \in \Lambda_p}|\Omega(\gamma)|e^{-R|Z_{\gamma}|}
    \end{equation}
    converges normally on compact subsets of $M$.
    
\end{itemize}
\end{definition}
\begin{remark} \leavevmode
\begin{itemize}
    \item The support property implies that if $\gamma\in \text{Supp}(\Omega)$, then $Z_{\gamma}\neq 0$. Furthermore, for any $R>0$ and $p\in M$, there can only be finitely many $Z_{\gamma}$ with $\gamma \in \Lambda_p\cap\text{Supp}(\Omega)$ and $|Z_{\gamma}|<R$. In particular, we must have $|Z_{\gamma}|\to \infty$ as $|\gamma|\to \infty$ for $\gamma \in  \Lambda_p\cap \text{Supp}(\Omega)$.
    \item  The convergence property is stronger than the one on \cite{VarBPS}. However, it will simplify technical details of convergence and term by term differentiation of sums that will appear below in the case that $\text{Supp}(\Omega)$ is infinite. We remark that the BPS indices appearing in the string theory setting of \cite{HMreview1,HMreview2,HMmetric} are not expected to satisfy even the weaker convergence condition on \cite{VarBPS}.
\end{itemize}
\end{remark}

We will only consider variation of BPS structures over an ASK manifold that are ``adapted" to the ASK structure in the following sense:

\begin{definition}\label{adapted}
Let $(M,g_M,\omega_M,\nabla,\Lambda)$ be an integral ASK manifold admitting a central charge $Z: M \to \Lambda^*\otimes \mathbb{C}$ (recall Proposition \ref{centralchargeprop} and \ref{centralchargeCASK}). Having fixed a central charge $Z$, an adapted variation of BPS structures over $(M,g_M,\omega_M,\nabla,\Lambda)$ is a variation of BPS structures $(M,\Lambda',Z',\Omega)$ such that $(\Lambda',Z')=(\Lambda,Z)$. In the case of a CASK manifold we always take the canonical central charge.
\end{definition}

 Now consider an adapted variation of BPS structures $(M,\Lambda,Z,\Omega)$ over $(M,g_M,\omega_M,\nabla,\Lambda)$. Furthermore, we consider the bundle $\pi:\mathcal{M}\to M$ of ``twisted" unitary characters given by
 
 \begin{equation}
     \mathcal{M}_u:=\{\theta : \Lambda_u \to \mathbb{R}/2\pi \mathbb{Z} \;\; | \;\; \theta_{\gamma+ \gamma'}= \theta_{\gamma}+\theta_{\gamma'}+\pi \langle \gamma,\gamma' \rangle\}\,,
 \end{equation}
 
 \begin{remark} The reason for considering twisted characters has to do with implementing the wall-crossing formalism from $\cite{KS}$. We remark that $\mathcal{M}$ and $\mathcal{N}$ (see (\ref{charbundle})) can be locally identified (non-canonically), but they might differ topologically (see \cite{GMN} for a discussion on this). 
 
 \end{remark}

 In \cite{GMN,NewCHK}, the proposed intanton corrected HK structure on $\mathcal{M}$ is then described as follows: first one must find locally defined functions $\mathcal{X}_{\gamma}:U\subset \mathcal{M}\times \mathbb{C}^{\times}\to \mathbb{C}^{\times}$, labeled by local sections $\gamma$ of $\Lambda|_{\pi(U)}$, and satisfying the ``GMN equations":
    
    \begin{equation}\label{GMNeq}                    \mathcal{X}_{\gamma}(\theta,\zeta)=
    \mathcal{X}^{\text{sf}}_{\gamma}(\theta,\zeta)\exp \Big[-\frac{1}{4\pi i}\sum_{\gamma'\in \Lambda_{\pi(\theta)}}\Omega(\gamma')\langle \gamma,\gamma' \rangle\int_{\mathbb{R}_{-}Z_{\gamma'}}\frac{d\zeta'}{\zeta'}\frac{\zeta'+\zeta}{\zeta'-\zeta}\log(1-\mathcal{X}_{\gamma'}(\theta,\zeta'))\Big]\,,
    \end{equation}
    where 
    \begin{equation}
        \mathcal{X}^{\text{sf}}_{\gamma}(\theta,\zeta)=\exp[\pi \zeta^{-1}Z_{\gamma}+i\theta_{\gamma}+\pi  \zeta \overline{Z}_{\gamma}]\,.
    \end{equation}
    For a fixed $\theta \in \mathcal{M}$, the functions $\mathcal{X}_{\gamma}(\theta,\zeta)$ have discontinuities in $\zeta$ along the rays $\mathbb{R}_{-}Z_{\gamma}$ with $\gamma \in \text{Supp}(\Omega)$ (the so-called BPS rays). Furthermore, a consequence of satisfying (\ref{GMNeq}) is that $\mathcal{X}_{\gamma+\gamma'}=(-1)^{\langle \gamma,\gamma' \rangle}\mathcal{X}_{\gamma}\mathcal{X}_{\gamma'}$, which is related to the twist in the unitary characters, and important for the wall-crossing formalism. \\
    
    Then one defines the following $\mathbb{C}^{\times}$-family of complex $2$-forms on $\mathcal{M}$:
    \begin{equation}
    \varpi(\zeta)=\frac{1}{8\pi^2 }\langle d\log(\mathcal{X}(\zeta))\wedge d\log(\mathcal{X}(\zeta))\rangle
    \end{equation}
    where $d$ differentiates only in the $\mathcal{M}$ directions. The discontinuities of $\mathcal{X}_{\gamma}(\zeta)$ in $\zeta$ turn out to not affect $\varpi(\zeta)$. Moreover, the KS wall-crossing formula is used to argue that $\varpi(\zeta)$ is actually well-defined over $\mathcal{W}\subset M$, where the BPS indices $\Omega(\gamma)$ jump. \\
    
 Finally, they argue that there is a hyperk\"{a}hler twistor space structure on $\mathcal{M}\times \mathbb{C}P^1$, whose $\mathcal{O}(2)$-twisted family of holomorphic symplectic forms is given by $\zeta \varpi(\zeta) \otimes \partial_{\zeta}$. In particular, if we expand $\zeta \varpi(\zeta)$ in $\zeta$, we obtain an expression of the form
    
    \begin{equation}
        \zeta\varpi(\zeta)=-\frac{i}{2}\varpi + \zeta\omega_3 -\frac{i}{2}\zeta^2 \overline{\varpi}\,,
    \end{equation}
    where $\varpi$ and $\omega_3$ give a holomorphic symplectic form and K\"{a}hler form with respect to one of the complex structures; and $\text{Re}(\varpi)$, $\text{Im}(\varpi)$ and $\omega_3$ give a triple of K\"{a}hler forms for the hyperk\"{a}hler structure.

\end{subsection}

\begin{subsection}{Mutually local variations of BPS structures and the instanton corrected HK structure}
\label{mutsec}

One of the main issues in describing the instanton corrected HK structure lies in solving the equations (\ref{GMNeq}).  Below, we will restrict to a case where the integral equations (\ref{GMNeq}) reduce to integral formulas, and write down the candidate $\varpi$ and $\omega_3$. We will then show in Theorem \ref{theorem1} under what conditions they define a HK structure on $N=T^*M/\Lambda^*$.

\begin{definition}\label{defMutLocBPS}
A variation of BPS structures $(M,\Lambda,Z,\Omega)$ is mutually local if $\gamma, \gamma' \in \text{Supp}(\Omega)$ implies that $\langle \gamma,\gamma' \rangle=0$.
\end{definition}
\begin{remark}\leavevmode
\begin{itemize}
    \item The mutually local condition implies $\mathcal{W}=\emptyset$, and hence no wall-crossing occurs for the BPS indices $\Omega(\gamma)$. In particular, given a local section $\gamma$ of $\Lambda$, $\Omega(\gamma)$ is a locally constant function on $M$.
    \item On the other hand, given $(M,g_M,\omega_M,\nabla,\Lambda)$ with an adapted mutually local variation of BPS structures,  the mutually local condition implies that we can find a local Darboux frame $(\widetilde{\gamma}_i,\gamma^i)$ of $\Lambda$ such that $\text{Supp}(\Omega)\subset \text{span}_{\mathbb{Z}}\{\gamma^i\}$ (see Lemma \ref{lemmatech} below).  It is then easy to see that the GMN equations (\ref{GMNeq}) for $\{\mathcal{X}_{\widetilde{\gamma}_i}(\zeta),\mathcal{X}_{\gamma^i}(\zeta)\}$ reduce from integral equations to integral formulas. In Lemma \ref{mutuallylocalforms} below we write down the corresponding candidate $\varpi$ and $\omega_3$ obtained from the explicit formulas for $\{\mathcal{X}_{\widetilde{\gamma}_i}(\zeta),\mathcal{X}_{\gamma^i}(\zeta)\}$ (see also \cite[section 4.3 and 5.6]{GMN}). 
\end{itemize}
\end{remark}

\begin{remark} 
For the rest of the paper we will only consider adapted variations of mutually local BPS structures over an integral ASK manifold $(M,g_M,\omega_M,\nabla, \Lambda)$ admitting a central charge (recall Propositions \ref{centralchargeprop}, \ref{centralchargeCASK}, and Definition \ref{adapted}). For simplicity, we will denote them just by $\Omega$ and refer to them as mutually local variations of BPS structures, omitting the word ``adapted". In the CASK case we always use the canonical central charge.

\end{remark}
Consider $(M,g_M,\omega_M,\nabla,\Lambda)$ with a mutually local variation of BPS structures $\Omega$. Following, \cite[section 4.3 and 5.6]{GMN}, we define the following forms on  $N=T^*M/\Lambda^*$ (we will omit from the notation pullbacks by the canonical projection $\pi:N\to M$):

\begin{lemma}
Consider a local section $\gamma$ of $\Lambda|_U$ for $U\subset M$ and with $\gamma \in \text{Supp}(\Omega)$. Let \begin{equation}\label{inst1}
        \begin{split}
            V_{\gamma}^{\text{inst}}&:=\frac{1}{2\pi}\sum_{n>0}e^{in\theta_{\gamma}}K_0(2\pi n|Z_{\gamma}|)\\
            A_{\gamma}^{\text{inst}}&:=-\frac{1}{4\pi}\sum_{n>0}e^{in\theta_{\gamma}}|Z_{\gamma}|K_1(2\pi n|Z_{\gamma}|)\Big( \frac{dZ_{\gamma}}{Z_{\gamma}}-\frac{d\overline{Z}_{\gamma}}{\overline{Z}_{\gamma}}\Big)
        \end{split}
    \end{equation}
where $K_0$ and $K_1$ are modified Bessel functions of the second kind. Then $V_{\gamma}^{\text{inst}}\in C^{\infty}(\pi^{-1}(U))$ and $A_{\gamma}^{\text{inst}}\in \Omega^1(\pi^{-1}(U))$.
\end{lemma}

\begin{proof}
The convergence of the two series on the right-hand side of  (\ref{inst1}) is compact normal on the set $\{p\in \pi^{-1}(U) \;\; | \;\; Z_\gamma(\pi(p)) \neq 0\}$, thanks to the asymptotics $K_{\nu}(x) \sim \sqrt{\frac{\pi}{2x}}e^{-x}(1+O(\frac{1}{x}))$ for $x\rightarrow \infty$, $\nu = 0,1$.  On the other hand, by the support property (\ref{supportproperty}) and our  assumption that $\gamma \in \text{Supp}(\Omega)$, we must have $\{p\in \pi^{-1}(U) \;\; | \;\; Z_\gamma(\pi(p)) \neq 0\}=\pi^{-1}(U)$. Hence, as a consequence of compact convergence $V_{\gamma}^{\text{inst}}$ defines a smooth function on $\pi^{-1}(U)$ and $A_{\gamma}^{\text{inst}}$ defines a smooth $1$-form on $\pi^{-1}(U)$. 
\end{proof}

The candidate instanton corrected $\varpi$ and $\omega_3$ are then the following:

\begin{lemma}\label{mutuallylocalforms} Let

\begin{equation} \label{holsym}
    \varpi:=-\frac{1}{2\pi}\langle dZ\wedge d\theta \rangle + \sum_{\gamma}\left( \Omega(\gamma)dZ_{\gamma}\wedge A_{\gamma}^{\text{inst}}  +\frac{i\Omega(\gamma)}{2\pi }V_{\gamma}^{\text{inst}}d\theta_{\gamma}\wedge dZ_{\gamma}\right)
\end{equation}
\begin{equation}\label{invKF}
    \omega_{3}:=\frac{1}{4}\langle dZ\wedge d\overline{Z}\rangle-\frac{1}{8\pi^2} \langle d\theta\wedge d\theta \rangle +\sum_{\gamma}\left( \frac{i\Omega(\gamma)}{2}V^{\text{inst}}_{\gamma}dZ_{\gamma}\wedge d\overline{Z}_{\gamma}+\frac{\Omega(\gamma)}{2\pi } d\theta_{\gamma}\wedge A_{\gamma}^{\text{inst}}\right)\,.
\end{equation}
Then $\varpi \in \Omega^2(N,\mathbb{C})$ and $\omega_3\in \Omega^2(N)$.
\end{lemma}

\begin{remark}
Notice that the sums over $\gamma$ in the definitions of $\varpi$ and $\omega_3$ are monodromy invariant due to the monodromy invariance of $\Omega$ (see Remark \ref{WKremark}), so they make global sense.
\end{remark}

\begin{proof}
If $\text{Supp}(\Omega)$ is finite, then clearly $\varpi$ and $\omega_3$ define smooth forms on $N$ by the previous lemma. Otherwise, consider a compact set $K\subset N$ such that we have a frame $(\widetilde{\gamma}_i,\gamma^i)$ of $\Lambda|_{\pi(K)}$, and a covariantly constant choice of norm for $\Lambda|_{\pi(K)}$. Writing (\ref{holsym}) and (\ref{invKF}) with respect to $dZ_{\widetilde{\gamma}_i}$ $dZ_{\gamma^i}$, $d\theta_{\widetilde{\gamma}_i}$ and $d\theta_{\gamma^i}$, we can use the support property (\ref{supportproperty}) to reduce the question of convergence of (\ref{holsym}) and (\ref{invKF}) to the convergence of sums of the form

\begin{equation}\label{sums}
    \sum_{\gamma}|\Omega(\gamma)||Z_{\gamma}|^2\sum_{n>0}K_{\nu}(2\pi n |Z_{\gamma}|)\,,
\end{equation}
where $\nu=0,1$. By the use of the asymptotics of the Bessel functions and the support property (\ref{supportproperty}), it is then easy to check that for any $0<\epsilon<1$ we can find $C_1>0$ and $C_2>0$ such for $|\gamma|>C_1$ we have

\begin{equation}
    |Z_{\gamma}|^2\sum_{n>0}K_{\nu}(2\pi n|Z_{\gamma}|)<C_2e^{-2\pi(1-\epsilon)|Z_{\gamma}|}
\end{equation}
(recall that by the support property $|Z_{\gamma}|\to \infty$ uniformly over $K$ as $|\gamma| \to \infty$ with $\gamma \in \text{Supp}(\Omega)$).
By the convergence property (\ref{convergenceproperty}), we then see that the infinite sums (\ref{sums}) converge normally over compact subsets of $N$, and hence the (\ref{holsym}) and (\ref{invKF}) define smooth forms on $N$.\\

The reality of $\omega_3$ follows from the identities $\overline{V_{\gamma}^{\text{inst}}}=V_{-\gamma}^{\text{inst}}$ and $\overline{A_{\gamma}^{\text{inst}}}=-A_{-\gamma}^{\text{inst}}$, together with the fact that $\Omega(\gamma)=\Omega(-\gamma)$. 
\end{proof}

The key notion that will guarantee
that the triple $(\omega_1:=\text{Re}(\varpi),\omega_2:=\text{Im}(\varpi),\omega_3)$ defines a HK structure on $N$ is the following: 

\begin{definition} \label{defT}Let $\pi:N\to M$ be the canonical projection, where $N=T^*M/\Lambda^*$. We will say that $(M,g_M,\omega_M,\nabla,\Lambda)$ and a mutually local $\Omega$ are compatible if the tensor field 

\begin{equation}\label{non-deg}
    T:=\pi^*g_{M}+ \sum_{\gamma}\Omega(\gamma)V_{\gamma}^{\text{inst}}\pi^*|dZ_{\gamma}|^2
\end{equation}
on $N$ is horizontally non-degenerate, i.e. $T$ is non-degenerate on the normal bundle of the fibers.  
\end{definition}

\begin{theorem}\label{theorem1} Consider an ASK manifold $(M,g_M,\omega_M,\nabla,\Lambda)$ together with a mutually local variation of BPS structures $\Omega$. Then the triple $(\omega_1=\text{Re}(\varpi),\omega_2=\text{Im}(\varpi),\omega_3)$ of real $2$-forms on $N$ define a pseudo-HK structure on $N$ if and only if $(M,g_M,\omega_M,\nabla,\Lambda)$ and $\Omega$ are compatible. 
\end{theorem}

Before doing the proof of Theorem \ref{theorem1}, we will need two lemmas. Lemma \ref{lemmatech} is a technical result needed for Lemma \ref{usefulexpressions}; while Lemma \ref{usefulexpressions} collects some useful expressions needed for the proof of Theorem \ref{theorem1}. 

\begin{lemma} \label{lemmatech} Let $\Lambda$ be a rank $2n$ lattice with a skew pairing $\omega:\Lambda \times \Lambda \to \mathbb{Z}$ admitting a Darboux basis, and let $S\subset \Lambda$ be a subset such that $\gamma, \gamma'\in S$ implies  $\omega(\gamma,\gamma')=0$. Then there is a Darboux basis $(\widetilde{\gamma}_i,\gamma^i)$ of $\Lambda$ such that $S\subset \text{span}_{\mathbb{Z}}\{\gamma^i\}$.

\end{lemma}
\begin{proof}
Let $L$ be a maximal isotropic  subgroup of $\Lambda$ such that $S\subset L$.  Since $\Lambda$ is a finitely generated free abelian group, the same is true for $L$, and $\text{rk}(L)\leq \text{rk}(\Lambda)=2n$. We 
claim that 
$\text{rk}(L)=n$. To see this, we consider the map $\psi_\omega :\Lambda \to \text{Hom}(L,\mathbb{Z})$ given by $\gamma \to \omega(\gamma,-)|_L$. Because $\omega$ admits a Darboux basis, 
$\psi_\omega :\Lambda \to \text{Hom}(L,\mathbb{Z})$ is surjective. 
Here we use that $L$ is primitive (by maximality) and thus the 
natural map $\text{Hom}(\Lambda,\mathbb{Z})\rightarrow \text{Hom}(L,\mathbb{Z})$ is surjective. This in turn follows from the fact that every primitive system of vectors in a lattice can be extended to a basis. Hence, if we denote $L^{\omega}:=\text{Ker}(\psi_\omega)$, we obtain $\text{rk}(L^{\omega})=\text{rk}(\Lambda)-\text{rk}(L)$. On the other hand, $L\subset L^{\omega}$ implies that $\text{rk}(L)\leq \text{rk}(L^{\omega})$, and since $L$ is a maximal isotropic sublattice containing $S$, we must have $\text{rk}(L)=\text{rk}(L^\omega)$. It 
follows that 
$\text{rk}(L)=n$. \\

Now let $\{\alpha^i\}_{i=1,..,n}$ be a basis for $L$. Using again that $L$ is a maximal isotropic sublattice, we have that 
$\alpha^1$ 
is primitive in $\Lambda$.
Now let $(\widetilde{\beta}_i,\beta^i)$ be a Darboux basis of $(\Lambda,\omega)$. Writing $\alpha^1=b^i\widetilde{\beta}_i+b_i\beta^i$ and using the fact that $\alpha^1$ is primitive, we must have $\text{gcd}\{b^i,b_i\}_{i=1}^n=1$. By Bezout's identity, there exist integers $\{a_i,a^i\}_{i=1}^n$ such that $a_ib^i + b_ia^i=1$. Defining

\begin{equation}
    \widetilde{\alpha}_1:= a^i\widetilde{\beta}_i -a_i\beta^i\,,
\end{equation}
we then see that $\omega(\widetilde{\alpha}_1,\alpha^1)=1$. We set $\gamma^1:=\alpha^1$ and $\widetilde{\gamma}_1:=\widetilde{\alpha}_1$.\\

Assume by induction that we have found $\{\widetilde{\gamma}_i,\gamma^i\}_{i=1,...,r}$ satisfying $\omega(\widetilde{\gamma}_i,\gamma^j)=\delta_i^j$, $\omega(\gamma^i,\gamma^j)=0$, $\omega(\widetilde{\gamma}_i,\widetilde{\gamma}_j)=0$ and such that $\text{span}(\gamma^1,...\gamma^r)=\text{span}(\alpha^1,...,\alpha^r)\subset L$. If $r=n=\text{rank}(L)$ we are done. Otherwise, pick $\alpha^{r+1}$ and define
\begin{equation}\label{gammar+1}
    \gamma^{r+1}:=\alpha^{r+1}+\sum_{i=1}^r\omega(\alpha^{r+1},\widetilde{\gamma}_i)\gamma^i \,.
\end{equation}
Then $\omega(\gamma^{r+1},\gamma^i)=0$, $\omega(\gamma^{r+1},\widetilde{\gamma}_i)=0$ and $\text{span}(\gamma^1,...,\gamma^{r+1})=\text{span} (\alpha^1,...,\alpha^{r+1})$. Furthermore, by (\ref{gammar+1}) and the fact that $L$ is maximal isotropic, we again have that $\gamma^{r+1}$ is primitive. We then use the Darboux frame $(\widetilde{\beta}_i,\beta^i)$  and the Bezout identity the same way as before to find $\widetilde{\alpha}_{r+1}$ such that $\omega(\widetilde{\alpha}_{r+1},\gamma^{r+1})=1$.\\

Define
\begin{equation}
    \widetilde{\gamma}_{r+1}:=\widetilde{\alpha}_{r+1}-\sum_{i=1}^r\omega(\widetilde{\alpha}_{r+1},\gamma^i)\widetilde{\gamma}_i +  \sum_{i=1}^r\omega(\widetilde{\alpha}_{r+1},\widetilde{\gamma}_{i})\gamma^i\,,
\end{equation}
we then have $\omega(\widetilde{\gamma}_{r+1},\gamma^{r+1})=1$ and for $i=1,...,r$ we have $\omega(\widetilde{\gamma}_{r+1},\gamma^i)=\omega(\widetilde{\gamma}_{r+1},\widetilde{\gamma}_i)=0$. Hence, we have found $\{\widetilde{\gamma}_i,\gamma^i\}_{i=1,...,r+1}$ satisfying $\omega(\widetilde{\gamma}_i,\gamma^j)=\delta_i^j$, $\omega(\gamma^i,\gamma^j)=0$, $\omega(\widetilde{\gamma}_i,\widetilde{\gamma}_j)=0$ and such that $\text{span}(\gamma^1,...,\gamma^{r+1})=\text{span} (\alpha^1,...,\alpha^{r+1})$.\\

At the n-th step we then have a Darboux frame $\{\widetilde{\gamma}_i,\gamma^i\}$ of $\Lambda$ such that $S\subset L=\text{span}(\gamma^1,...,\gamma^n)$, which is what we wanted.
\end{proof}

\begin{lemma}\label{usefulexpressions}
Let $(M,g_M,\omega_M,\nabla,\Lambda)$ and $\Omega$ be as before. Given a local Darboux frame $(\widetilde{\gamma}_i,\gamma^i)$ for $\Lambda$ such that $\text{Supp}(\Omega)\subset \text{span}_{\mathbb{Z}}\{\gamma^i\}$ (see Lemma \ref{lemmatech}), we write $\gamma=n_i(\gamma)\gamma^i$ for $\gamma \in \text{Supp}(\Omega)$ and define the complex $1$-forms on $N$ 
\begin{equation}\label{instW}
    W_i:=d\theta_{\widetilde{\gamma}_i}-\tau_{ij}d\theta_{\gamma^j}, \;\;\;\;\;\; W_i^{\text{inst}}:=\sum_{\gamma}\Omega(\gamma)n_i(\gamma)(2\pi A_{\gamma}^{\text{inst}}-iV_{\gamma}^{\text{inst}}d\theta_{\gamma}),\;\;\;\;\;\; Y_i:=W_i+W_i^{\text{inst}}\,,
\end{equation}
where $dZ_{\widetilde{\gamma}_i}=\tau_{ij}dZ_{\gamma^j}$. Furthermore, we define the matrices
\begin{equation}\label{instN}
        N_{ij}:=\text{Im}{\tau}_{ij}, \;\;\;\;\;\; N_{ij}^{\text{inst}}:=\sum_{\gamma}\Omega(\gamma)V^{\text{inst}}_{\gamma}n_i(\gamma)n_j(\gamma), \;\;\;\;\;\;
        M_{ij}:=N_{ij}+N_{ij}^{\text{inst}}\,.
\end{equation}
We then have 
\begin{equation}\label{omega12}
    \varpi=\frac{1}{2\pi}dZ_{\gamma^i}\wedge Y_i, \;\;\;\;\omega_1=\frac{1}{4\pi}(dZ_{\gamma^i}\wedge Y_i + d\overline{Z}_{\gamma^i}\wedge \overline{Y}_i),\;\;\;\;
    \omega_2=\frac{1}{4\pi i}(dZ_{\gamma^i}\wedge Y_i - d\overline{Z}_{\gamma^i}\wedge \overline{Y}_i)\,,
\end{equation}
and assuming $\Omega$ is compatible with the ASK manifold,
\begin{equation}\label{KF3}
    \omega_3=\frac{i}{2}M_{ij}dZ_{\gamma^i}\wedge d\overline{Z}_{\gamma^j}+\frac{i}{8\pi^2}M^{ij}Y_i\wedge \overline{Y}_j\,.
\end{equation}
\end{lemma}
\begin{remark} Note that the symmetric matrix $M_{ij}$ is real, in virtue of the 
property $\Omega (-\gamma ) = \Omega  (\gamma)$ and $\overline{V_{\gamma}^{\text{inst}}}=V_{-\gamma}^{\text{inst}}$. On the other hand, the fact that $M_{ij}$ is invertible follows from the compatibility condition (\ref{non-deg}) written with respect to $(\widetilde{\gamma}_i,\gamma^i)$.
\end{remark}
\begin{proof} Using that $dZ_{\widetilde{\gamma}_i}=\tau_{ij}dZ_{\gamma^j}$ we obtain
\begin{equation}
    \begin{split}
        \varpi&=-\frac{1}{2\pi}\langle dZ\wedge d\theta \rangle + \sum_{\gamma}\Omega(\gamma)dZ_{\gamma}\wedge A_{\gamma}^{\text{inst}}  +\frac{i\Omega(\gamma)}{2\pi }V_{\gamma}^{\text{inst}}d\theta_{\gamma}\wedge dZ_{\gamma}\\
        &=\frac{1}{2\pi}dZ_{\gamma^i}\wedge (d\theta_{\widetilde{\gamma}_i}-\tau_{ij}d\theta_{\gamma^j})+\frac{1}{2\pi}dZ_{\gamma^i}\wedge \Big(\sum_{\gamma}\Omega(\gamma)n_i(\gamma)(2\pi A_{\gamma}^{\text{inst}}-iV_{\gamma}^{\text{inst}}d\theta_{\gamma})\Big)\\
        &=\frac{1}{2\pi}dZ_{\gamma^i}\wedge Y_i\,.
    \end{split}
\end{equation}
The formulas for $\omega_1$ and $\omega_2$ then follow from $\omega_1=\text{Re}(\varpi)$ and $\omega_2=\text{Im}(\varpi)$.\\

On the other hand, to see that (\ref{KF3})  holds it is easy to check that:

\begin{equation}
    \frac{i}{2}M_{ij}dZ_{\gamma^i}\wedge d\overline{Z}_{\gamma^j}=\frac{1}{4}\langle dZ\wedge d\overline{Z} \rangle +\sum_{\gamma}\frac{i\Omega(\gamma)}{2}V^{\text{inst}}_{\gamma}dZ_{\gamma}\wedge d\overline{Z}_{\gamma}\,,
\end{equation}
while for the second term,  after expanding the terms in $Y_i$, using the fact that $M_{ij}$ is symmetric, and reorganizing terms we obtain
\begin{equation}
\begin{split}
    \frac{i}{8\pi^2}M^{ij}Y_i\wedge \overline{Y}_j&=\frac{i}{8\pi^2}M^{ij}\Big(2iM_{ik}[d\theta_{\widetilde{\gamma}_j}\wedge d\theta_{\gamma^k}+\sum_{\gamma}\Omega(\gamma)n_j(\gamma)2\pi A_{\gamma}^{\text{inst}}\wedge d\theta_{\gamma^k}+\text{Re}(\tau_{jl})d\theta_{\gamma^k}\wedge d\theta_{\gamma^l}]\Big)\\
    &=-\frac{1}{4\pi^2}\Big[d\theta_{\widetilde{\gamma}_k}\wedge d\theta_{\gamma^k}+\sum_{\gamma}\Omega(\gamma)2\pi A_{\gamma}^{\text{inst}}\wedge d\theta_{\gamma}+\text{Re}(\tau_{kl})d\theta_{\gamma^k}\wedge d\theta_{\gamma^l}\Big]\\
    &=-\frac{1}{8\pi^2}\langle d\theta \wedge d\theta \rangle+ \sum_{\gamma}\frac{\Omega(\gamma)}{2\pi}d\theta_{\gamma} \wedge A_{\gamma}^{\text{inst}}\,,
\end{split}
\end{equation}
so that (\ref{KF3}) holds.
\end{proof}
\begin{proof} (of Theorem $\ref{theorem1}$) We start by showing that the compatibility assumption (\ref{non-deg}) implies the non-degeneracy of the $\omega_{\alpha}$ forms. We fix a local Darboux frame $(\widetilde{\gamma}_i,\gamma^i)$ with $\text{Supp}(\Omega)\subset \text{span}\{\gamma^i\}$ and denote for simplicity $Z^i:=Z_{\gamma^i}$. We work locally with the real frame $\partial_{x^i}:=\partial_{Z^i}+\partial_{\overline{Z}^i}$, $\partial_{u^i}:=i(\partial_{Z^i}-\partial_{\overline{Z}^i})$, $\partial_{\theta^i}:=\partial_{\theta_{\gamma^i}}$ and $\partial_{\widetilde{\theta}_i}:=\partial_{\theta_{\widetilde{\gamma}_i}}$ and recall that the compatibility condition (\ref{non-deg}) implies the invertibility of the matrix $M_{ij}$.

\begin{itemize}
    \item Non-degeneracy of $\omega_1$: we write below relations that are sufficient to deduce non-degeneracy. Using (\ref{instW}) and (\ref{omega12}) we obtain
    \begin{equation}\label{non-deg-omega1}
        \begin{split}
     \omega_1(\partial_{\widetilde{\theta}_i},\partial_{x^j})&=-\frac{\delta_{ij}}{2\pi}, \;\;\;\;\;\;\;\omega_1(\partial_{\theta^i}, \partial_{u^j})=-\frac{1}{2\pi}M_{ij}\\
     \omega_1(\partial_{\widetilde{\theta}_i},\partial_{\theta^j})&=\omega_1(\partial_{\theta^i},\partial_{\theta^j})=\omega_1(\partial_{\widetilde{\theta}_i},\partial_{\widetilde{\theta}_j})=\omega_1(\partial_{u^i},\partial_{\widetilde{\theta}_j})=0\, .
        \end{split}
    \end{equation}
    So $\omega_1$ is represented by a block triangular matrix (with respect to the second diagonal) with invertible diagonal blocks. Hence, we conclude that $\omega_1$ is non-degenerate.
    
    \item Non-degeneracy of $\omega_2$: as before, we write below sufficient relations to deduce non-degeneracy. Using (\ref{instW}) and (\ref{omega12})
    \begin{equation}
        \begin{split}
         \omega_2(\partial_{\widetilde{\theta}_i},\partial_{u^j})&=-\frac{\delta_{ij}}{2\pi}, \;\;\;\;\;\;\;
         \omega_2(\partial_{\theta^i},\partial_{x^j})=\frac{1}{2\pi}M_{ij}\\
        \omega_2(\partial_{\widetilde{\theta}_i},\partial_{\theta^j})&=\omega_2(\partial_{\theta^i},\partial_{\theta^j})=\omega_2(\partial_{\widetilde{\theta}_i},\partial_{\widetilde{\theta}_j})=\omega_2(\partial_{x^i},\partial_{\widetilde{\theta}_j})=0\, .
        \end{split}
    \end{equation}
   So $\omega_2$ is represented by a block triangular matrix with invertible diagonal blocks (of size $2\dim_\mathbb{R} (M)$ on the second diagonal). This shows that $\omega_2$ is nondegenerate.
   
    \item Non-degeneracy of $\omega_3$: using (\ref{invKF}) and (\ref{instN}) we obtain the following sufficient relations to show non-degeneracy
    \begin{equation}
        \begin{split}
         \omega_3(\partial_{\widetilde{\theta}_i}, \partial_{\theta^j})&=-\frac{\delta_{ij}}{4\pi^2}, \;\;\;\;\;\;\; \omega_3(\partial_{x^i},\partial_{u^j})= M_{ij}\\
        \omega_3(\partial_{u^i},\partial_{u^j})&=\omega_3(\partial_{x^i},\partial_{\widetilde{\theta}_j})=\omega_3(\partial_{u^i},\partial_{\widetilde{\theta}_j})=\omega_3(\partial_{\widetilde{\theta}_i},\partial_{\widetilde{\theta}_j})=0 \,.
        \end{split}
    \end{equation}
    Using the block structure of the representing matrix we see that it can be transformed to block diagonal form with invertible diagonal blocks using row and column operations. This proves that it has maximal rank.
   
\end{itemize}

Now we check that
\begin{equation}
    d\omega_{\alpha}=0 \;\;\;\; \alpha=1,2,3 \,.
\end{equation}

For $\alpha=1,2$ it is enough to check that $d\varpi=0$, since $\varpi=\omega_1+i\omega_2$, and $\omega_1$ and $\omega_2$ are real. Using (\ref{holsym}), and the fact that the semi-flat part is closed, we have

\begin{equation}
    \begin{split}
    d\varpi=&\sum_{\gamma}\Omega(\gamma)\Big[\frac{i}{2\pi }dV_{\gamma}^{\text{inst}}\wedge d\theta_{\gamma} \wedge dZ_{\gamma}-dZ_{\gamma}\wedge dA_{\gamma}^{\text{inst}}\Big]\\
    =&\sum_{\gamma}\Omega(\gamma)\Big[-\frac{i}{4\pi}\sum_{n>0}ne^{in\theta_{\gamma}}K_1(2\pi n  |Z_{\gamma}|)\frac{Z_{\gamma}}{|Z_{\gamma}|}d\overline{Z}_{\gamma}\wedge d\theta_{\gamma}\wedge dZ_{\gamma} \\
    &-\frac{i}{4\pi}\sum_{n>0}ne^{in\theta_{\gamma}}K_1(2\pi n  |Z_{\gamma}|)\frac{|Z_{\gamma}|}{\overline{Z}_{\gamma}}dZ_{\gamma}\wedge d\theta_{\gamma}\wedge d\overline{Z}_{\gamma}\Big]\\
    =& \;\; 0
    \end{split}
\end{equation}
where we used that $K_0'=-K_1$.
On the other hand, using (\ref{invKF}) and the fact that the semi-flat part is closed, we obtain

\begin{equation}
    \begin{split}
        d\omega_3=& \sum_{\gamma}\Omega(\gamma)\Big[ \frac{i}{2}dV_{\gamma}^{\text{inst}}\wedge dZ_{\gamma}\wedge d \overline{Z}_{\gamma} - \frac{1}{2\pi} d\theta_{\gamma}\wedge dA_{\gamma}^{\text{inst}}\Big] \\
        =&\sum_{\gamma}\Omega(\gamma)\Big[ -\frac{1}{4\pi}\sum_{n>0}ne^{in\theta_{\gamma}}K_0(2\pi n |Z_{\gamma}|)d\theta_{\gamma}\wedge dZ_{\gamma}\wedge d\overline{Z}_{\gamma} \\
        &+ \frac{1}{4\pi}d\theta_{\gamma}\wedge\Big(\sum_{n>0}ne^{in\theta_{\gamma}}K_0(2\pi n |Z_{\gamma}|) dZ_{\gamma}\wedge d\overline{Z}_{\gamma}\Big)\Big]\\
        =&\;\; 0\end{split}
\end{equation}
where we used the identity $(xK_1(x))'=-xK_0(x)$.\\

We now define three endomorphism fields by
\begin{equation}
    I_{\alpha}=-\omega_\beta^{-1}\circ \omega_\gamma
\end{equation}
where $(\alpha,\beta,\gamma)$ are cyclically ordered. We will compute their action on a local frame of $TN\otimes \mathbb{C}$ and check that $I_1I_2=I_3$ and $I_{\alpha}^2=-1$.\\

We have a local frame of $T^*N\otimes \mathbb{C}$ given by $\{dZ^i, d\overline{Z}^i, Y_i,\overline{Y}_i\}$ (see Lemma \ref{usefulexpressions}). We consider the dual frame of vector fields $\{A_i,\overline{A}_i,B^i,\overline{B}^i\}$. Using (\ref{omega12}) and (\ref{KF3}) we find the following identities 

\begin{equation}
    \begin{split}
        \omega_1(A_i,-)=\frac{1}{4\pi}Y_i &\;\;\;\;\; \omega_1(B^i,-)=-\frac{1}{4\pi}dZ^i\\
        \omega_2(A_i,-)=\frac{1}{4\pi i}Y_i &\;\;\;\;\; \omega_2(B^i,-)=-\frac{1}{4\pi i}dZ^i\\
        \omega_3(A_i,-)=\frac{i}{2}M_{ij}d\overline{Z}^j & \;\;\;\;\; \omega_3(B^i,-)=\frac{i}{8\pi^2}M^{ij}\overline{Y}_j\,.
    \end{split}
\end{equation}
The evaluation on the conjugate elements of the frame give the conjugate of the corresponding expressions (recall that $(M_{ij})$ is a real symmetric matrix).\\

This allows us to compute the action of $I_{\alpha}$ on the frame, giving the following 
\begin{equation}\label{actcomplexstr}
    \begin{split}
        I_1(A_i)=2\pi M_{ij}\overline{B}^j &\;\;\;\;\; I_1(B^i)= -\frac{1}{2\pi}M^{ij}\overline{A}_j\\
        I_2(A_i)=-2\pi i M_{ij}\overline{B}^j & \;\;\;\;\; I_2(B^i)=\frac{i}{2\pi}M^{ij}\overline{A}_j\\
        I_3(A_i)=iA_i & \;\;\;\;\; I_3(B^i)=iB^i\,,
    \end{split}
\end{equation}
where again the value on the conjugate elements is given by taking conjugates on the RHS of the above equalities. From (\ref{actcomplexstr}) it is clear that
\begin{equation}
    I_1I_2=I_3 \;\;\;\;\;\; I_{\alpha}^2=-1 \;\;\; \text{for} \;\;\; \alpha=1,2,3\,,
\end{equation}
so each $I_{\alpha}$ defines an almost complex structure on $N$, and they satisfy the quaternion relations. These relations imply that the tensor field $\omega_\alpha \circ I_\alpha = \omega_\alpha (I_\alpha -,- )$ is independent of $\alpha$ and that the pseudo-Riemannian metric $g_N:=-\omega_{\alpha}\circ I_{\alpha}$ satisfies $I_{\alpha}^*g_N=g_N$ and $g_N(I_{\alpha}-,-)=\omega_{\alpha}(-,-)$.\\ 

Hence, we conclude that $(N,g_N,I_1,I_2,I_3)$ is almost pseudo-hyperk\"{a}hler, with $d\omega_{\alpha}=0$. By  Hitchin's lemma \cite[Lemma 6.8]{SDeq}, $(N,g_N,I_1,I_2,I_3)$ is then pseudo-hyperk\"{a}hler. \\

We have shown the sufficiency of the compatibility condition (\ref{non-deg}). The necessity follows from the previous identities used to show non-degeneracy of $\omega_{\alpha}$. Indeed, if $T$ in (\ref{non-deg}) is horizontally degenerate for some $p\in N$, then with respect to a Darboux frame $(\widetilde{\gamma}_i,\gamma^i)$ of $\Lambda$ around $\pi(p)\in M$ with $\text{Supp}(\Omega)\subset \text{span}\{\gamma^i\}$ the corresponding $M_{ij}$ is not invertible at $p$, and then the forms $\omega_{\alpha}$ can be shown to be degenerate at $p$.
\end{proof}

\begin{corollary}\label{HKmetriccoord} Let $(N,g_N,I_1,I_2,I_3)$ and $\{dZ_{\gamma^i}, d\overline{Z}_{\gamma^i}, Y_i,\overline{Y}_i\}$ be the pseudo-HK manifold and local frame of $T^*N\otimes \mathbb{C}$ from the previous theorem. Then in such a frame the metric $g_N$ has the local form  
\begin{equation}\label{HKmetric}
\begin{split}
    g_N&=dZ_{\gamma^i}M_{ij}d\overline{Z}_{\gamma^j}+\frac{1}{4\pi^2}Y_iM^{ij} \overline{Y}_j\\
    &=dZ_{\gamma^i}(N_{ij}+N_{ij}^{\text{inst}})d\overline{Z}_{\gamma^j}+\frac{1}{4\pi^2}(W_i+W_i^{\text{inst}})(N+N^{\text{inst}})^{ij}(\overline{W}_j+\overline{W}_j^{\text{inst}}) \,.  
\end{split}
\end{equation}
\end{corollary}
\begin{proof}
The first equality follows from (\ref{KF3}) together with the fact that $dZ_{\gamma^i}$ and $Y_i$ are $(1,0)$ forms with respect to $I_3$ due to (\ref{actcomplexstr}). The second equality follows from the formulas in Lemma \ref{usefulexpressions}.
\end{proof}
\begin{remark} the semi-flat HK metric corresponds in this case to

\begin{equation}
    g^{\text{sf}}=dZ_{\gamma^i}N_{ij}d\overline{Z}_{\gamma^j}+\frac{1}{4\pi^2}W_iN^{ij} \overline{W}_j\,.
\end{equation}
This expression shows that the torus bundle 
$\pi: (N,g^{\text{sf}})\rightarrow 
(M,g_M)$ is a Riemannian submersion with totally geodesic horizontal distribution and flat fibers. Due to the latter property, the metric $g^{\text{sf}}$ is called semi-flat.\\

This picture is lost when including the instanton corrections. Indeed, the horizontal part of the metric $g_N$ is no longer basic, since the functions $N_{ij}^{\text{inst}}$ are not basic. Moreover, the metric on the fibers is no longer translational invariant and the coordinate vector fields 
$\frac{\partial}{\partial Z_{\gamma^i}}, \frac{\partial}{\partial \overline{Z}_{\gamma^i}}$ are not  
horizontal, i.e.\ perpendicular to the fibers.

\end{remark}
\subsubsection{The CASK case and the rotating action}

We now assume that we start with an integral CASK manifold $(M,g_{M},\omega_M,\nabla,\xi,\Lambda)$, together with a compatible, mutually-local $\Omega$. We denote by $(N,g_N,\omega_1,\omega_2,\omega_3)$ the corresponding pseudo-HK manifold from Theorem \ref{theorem1}.

\begin{definition}
Let $(N,g_N,\omega_1,\omega_2,\omega_3)$ be a pseudo-HK manifold. An infinitesimal rotating circle action is a Killing vector field $V$ on $N$ such that $\mathcal{L}_V(\omega_1+i\omega_2)=i(\omega_1+i\omega_2)$ and $\mathcal{L}_V\omega_3=0$.
\end{definition}

In \cite[Section 3]{Conification} it is shown that one can lift $J\xi$ from $M$ to $N$, in such a way we get an infinitesimal rotating circle action for the semi-flat HK structure on $N$. The lift $V$ of $J\xi$ is defined as follows: if $q^{i}$ are special affine coordinates for $M$ and $(\pi^*q^i,p_i)$ the corresponding coordinates on $T^*M$, then 

\begin{equation}\label{lift}
    V(\pi^* q^i)=\pi^*\left(J\xi (q^i)\right), \;\;\;\;\; V(p_i)=0\,.
\end{equation}
This local definition does not depend on the choice of special affine coordinates, and hence defines global lift $V$ of $J\xi$. The vector field $V$ is, in fact, the $\nabla$-horizontal lift of $J\xi$ and furthermore descends to $N=T^*M/\Lambda^*$. 

\begin{proposition}\label{rotatingprop} 
The vector field $V$ defines an infinitesimal rotating circle action for the instanton corrected HK structure $(N,g_N,\omega_1,\omega_2,\omega_3)$.

\end{proposition}
\begin{proof}
The vector field $V$ of $N$ satisfies the following:

\begin{itemize}
    \item The central charge $Z$ satisfies $\mathcal{L}_V\pi^*Z_{\gamma}=\mathcal{L}_{J\xi}Z_{\gamma}=iZ_{\gamma}$ (recall Proposition \ref{centralchargeCASK} and the remark below Definition \ref{conicalholspecialcoords}). 
    Here and in the following we have omitted $\pi^*$ on the right-hand side for ease of notation, identifying function on $M$ with functions on $N$ via pull-back.
    \item The angle coordinates $\theta_{\gamma}$ are invariant by (\ref{lift}), i.e $\mathcal{L}_V\theta_{\gamma}=0$.
\end{itemize}

From the previous two points and (\ref{inst1}) we conclude that $\mathcal{L}_VV_{\gamma}^{\text{inst}}=0$ and $\mathcal{L}_VA_{\gamma}^{\text{inst}}=0$. Hence, from (\ref{holsym}) and (\ref{invKF}) we see that $\mathcal{L}_V\varpi=i\varpi$ while $\mathcal{L}_V\omega_3=0$. This implies that $\mathcal{L}_Vg_N=0$ by differentiating the identity  $g_N=\omega_3\circ \omega_1^{-1}\circ \omega_2$ and using the skew-symmetry of  $\omega^{-1}_{\alpha}\circ \omega_{\beta}$ in the indices $(\alpha,\beta)$. We conclude that $V$ defines an infinitesimal rotating circle action for the HK structure $(N,g_N,\omega_1,\omega_2,\omega_3)$.
\end{proof}

 Using our local expression (\ref{HKmetric}) of $g_N$, we see that we can write $g_N$ in the local real frame $dx^i:=\text{Re}(dZ_{\gamma^i})$, $du^i:=\text{Im}(dZ_{\gamma^i})$, $\alpha_i:=\text{Re}(Y_i)$ and $\beta_i:=\text{Im}(Y_i)$ as
 
 \begin{equation}
     g=M_{ij}(dx^idx^j + du^idu^j) +\frac{1}{4\pi^2}M^{ij}(\alpha_i\alpha_j + \beta_i\beta_j)\,.
 \end{equation}
 In particular, we see that if the (real) matrix $M_{ij}$ has signature $(n,m)$ then $g$ has signature $\sigma(g)=(4n,4m)$. In the CASK case we can say more:
 
\begin{proposition}\label{signatureHK} Let $(N,g_N,\omega_1,\omega_2,\omega_3)$ be an instanton corrected HK metric associated to a connected integral CASK manifold $(M,\omega,g_{M},\nabla,\xi,\Lambda)$ with a compatible $\Omega$. If the flow of $\xi$ induces a free-action on $M$ of the (multiplicative) monoid $\mathbb{R}_{\geq 1}$, then $\sigma(g)=\sigma(g^\text{sf})=(4,4n)$, where $\text{dim}_{\mathbb{C}}(M)=n+1$.
 
\end{proposition}
 
\begin{proof} 
By the connectedness assumption, it is enough to check the signature at a single point. We fix $p\in M$ and consider the ray $l_p:=\mathbb{R}_{\geq 1}\cdot p\subset M$ obtained by the free $\mathbb{R}_{\geq 1}$-action on $M$, and a Darboux frame $(\widetilde{\gamma}_i,\gamma^i)$ of $\Lambda$ with $\text{Supp}(\Omega)\subset \text{span}\{\gamma^i\}$ along $l_p$. With respect to this frame we have the matrix $M_{ij}$ controlling $\sigma(g)$:

\begin{equation}
    M_{ij}=\text{Im}(\tau_{ij}) + \sum_{\gamma}\Omega(\gamma)V_{\gamma}^{\text{inst}}n_i(\gamma)n_j(\gamma)\,.
\end{equation}
We consider the $\nabla$-horizontal lift of $\xi$ to $N$. Such a lift generates a free $\mathbb{R}_{\geq 1}$-action on $\pi^{-1}(l_p)\subset N$ that scales $Z_{\gamma}$ and leaves $\theta_{\gamma}$ invariant. The term $N_{ij}=\text{Im}(\tau_{ij})$ is invariant under the $\mathbb{R}_{\geq 1}$-action scaling the central charge (since $M$ is CASK), while the terms $V_{\gamma}^{\text{inst}}$ are exponentially decreasing as we act with a sufficiently big $t\in \mathbb{R}_{\geq 1}$ (due to the asymptotics of the Bessel functions). Hence, by the use of the convergence property (\ref{convergenceproperty}) and the $\mathbb{R}_{\geq 1}$-action we can make $M_{ij}$ as close as we want to $N_{ij}$ for an appropriate point in $\pi^{-1}(l_p)$. Hence, $\sigma(g_N)=\sigma(g^{\text{sf}})=2\sigma(g_{M})=(4,4n)$.
\end{proof}

\end{subsection}

\begin{section}{HK-QK correspondence for the corrected HK metric}
\label{hyperholsec}
Consider an integral CASK manifold $(M,g_{M},\omega_M,\nabla,\xi,\Lambda)$  together with a compatible, mutually local $\Omega$. From the results of Theorem \ref{theorem1} and Proposition \ref{rotatingprop} we obtain an HK manifold $(N,g_N,\omega_1,\omega_2,\omega_3)$ with an infinitesimal rotating action given by the vector field $V$. In this section we apply to $(N,g_N,\omega_1,\omega_2,\omega_3)$ the explicit description of the HK-QK  correspondence found in \cite[Theorem 2]{QKPSK}, to obtain a possibly indefinite QK manifold $(\overline{N},g_{\overline{N}})$.\\

One of the things that we will require in order to apply the HK-QK correspondence to $(N,g_N,\omega_1,\omega_2,\omega_3)$, is an $S^1$-principal bundle $\pi_N:P\to N$ having a connection $\eta \in \Omega^1(P)$ with curvature $F\in \Omega^2(N)$ satisfying

\begin{equation}\label{hyperhol}
    F=2\pi(\omega_3 -d(\iota_Vg_N)) \,.
\end{equation}

\begin{remark}\label{hyperholremark} One can show that $F$ is of type $(1,1)$ with respect to the triple of complex structures $I_{\alpha}$ for $\alpha=1,2,3$ (see for example \cite[Proposition 1]{HitHKQK}). Hence,  the principal circle bundle $(\pi_N:P\to N,\eta)$ is hyperholomorphic, in the sense that the complex line bundle associated to the defining representation of $U(1)$ is holomorphic for all three $I_\alpha$. \\

After constructing the required hyperholomorphic circle bundle in Section \ref{HHbundle}, we apply the HK-QK correspondence in Section \ref{HK-QK}. The main result of this section is Theorem \ref{theorem2}.

\end{remark}

\begin{subsection}{The hyperholomorphic circle bundle}\label{HHbundle}

In this section we construct the hyperholomorphic bundle $(\pi_N:P\to N,\eta)$. We use the notation $\pi_M: N\to M$ for the natural projection.

\begin{proposition}\label{linebundle} There exists an $S^1$-principal bundle $\pi_N:P\to N$ with connection $\Theta$ having curvature
\begin{equation}\label{curvaturetheta}
    d\Theta=-\frac{1}{4\pi}\pi^*_N \langle d\theta\wedge d\theta \rangle \,.
\end{equation}
\end{proposition}
\begin{proof}
We start by considering the trivial $\mathbb{R}$-principal bundle $T^*M\times \mathbb{R}\to T^*M$. We denote by $p:T^*M \to \Lambda^*\otimes \mathbb{R}$ the evaluation map. Hence, if we are given a local trivialization $(\widetilde{\gamma}_i,\gamma^i)$ of $\Lambda$, then $(p_{\widetilde{\gamma}_i},p_{\gamma^i})$ gives coordinates for the fibers of $T^*M$. On $T^*M\times \mathbb{R} \to T^*M$ we define the connection

\begin{equation}
    \Theta:=d\sigma - \frac{1}{4\pi}\langle p,dp \rangle \,,
\end{equation}
where $\sigma$ is a global coordinate on the $\mathbb{R}$-fiber.\\

We now define the bundle of discrete Heisenberg groups $\text{Heis}(\Lambda^*)\to M$ where $\text{Heis}(\Lambda^*):= 2\pi \Lambda^* \times \pi \mathbb{Z}$ and the group structure on the fibers is given by

\begin{equation}
    (2\pi \delta, \pi k)\cdot (2\pi \delta', \pi k')=(2\pi (\delta + \delta'), \pi (k+k' + \langle \delta, \delta' \rangle)) \,.
\end{equation}

We define a fiber-wise action (as bundles over $M$) of $\text{Heis}(\Lambda^*)$ on $T^*M\times \mathbb{R}$  by 

\begin{equation}
    (2\pi \delta,\pi k) \cdot (p,\sigma) = (p + 2\pi \delta, \sigma + \pi k +\frac{1}{4\pi}\langle 2\pi \delta, p \rangle )\,,
\end{equation}
and on $T^*M$ by
\begin{equation}
    (2\pi\delta, \pi k) \cdot p = p+2\pi \delta\,.
\end{equation}

With these actions the projection $T^*M\times \mathbb{R}\to T^*M$ is clearly equivariant and $\Theta$ is invariant under the action on $T^*M\times \mathbb{R}$. By taking the quotient we obtain an $S^1$-principal bundle $\pi_N:P\to N$, and the connection descends to a connection $\Theta$ on $P$. It is given by 
\begin{equation}
    \Theta= d\sigma -\frac{1}{4\pi}\pi^*_N\langle \theta,d\theta \rangle \,,
\end{equation}
where the  coordinate $\sigma$ is now periodic with period $\pi$ and transforms as 
$\sigma \mapsto \sigma'=\sigma +\frac{1}{2}\pi^*_N\langle \delta,\theta \rangle$ under translations $\theta \mapsto \theta'=\theta+2\pi \delta$, $\delta \in \Lambda^*$, 
in the sense that $(\theta , \sigma)$ and $(\theta',\sigma')$ describe the same point in the fiber of $P$.\end{proof}

The following lemma gives an expression for $\omega_3$ that will be convenient for defining the connection $\eta$ on $\pi_N:P\to N$ with curvature (\ref{hyperhol}).
\begin{lemma} For $\gamma$ a section of $\Lambda|_U\cap \text{Supp}(\Omega)$ with $U\subset M$, let $\eta_{\gamma}^{\text{inst}}\in \Omega^{1}(\pi_M^{-1}(U))$ be defined by
\begin{equation}\label{etagammainst}
    \eta_{\gamma}^{\text{inst}}:=\frac{i}{8\pi^2}\sum_{n>0}\frac{e^{in\theta_{\gamma}}}{n}|Z_{\gamma}|K_1(2\pi n|Z_{\gamma}|)\Big(\frac{dZ_{\gamma}}{Z_{\gamma}}-\frac{d\overline{Z}_{\gamma}}{\overline{Z}_{\gamma}}\Big)\,.
\end{equation}
Letting $r^2:=g_{M}(\xi,\xi)$, we can then write $\omega_3$ as follows:
\begin{equation}\label{omega3int}
    \omega_3=\frac{i}{2}\partial\overline{\partial}r^2 -\frac{1}{8\pi^2}\langle d\theta \wedge d\theta \rangle + d\Big(\sum_{\gamma}\Omega(\gamma)\eta_{\gamma}^{\text{inst}}\Big)\,.
\end{equation}

\end{lemma}
\begin{remark} Notice that $\sum_{\gamma}\Omega(\gamma)\eta_{\gamma}^{\text{inst}}$ makes global sense due to the monodromy invariance of $\Omega$. Furthermore, $\sum_{\gamma}\Omega(\gamma)\eta_{\gamma}^{\text{inst}}\in \Omega^{1}(N)$ by the same arguments of Lemma \ref{mutuallylocalforms}.

\end{remark}
\begin{proof}
Since $M$ is a CASK manifold, we have from (\ref{globalKP}):

\begin{equation}\label{inst3}
    \omega_M=\frac{1}{4}\langle dZ\wedge d\overline{Z}\rangle=\frac{i}{2}\partial\overline{\partial} r^2 \,.
\end{equation}
 On the other hand,  using the identity $(xK_1(x))'=-xK_0(x)$ we see that the instanton correction terms of $\omega_3$ satisfy
\begin{equation}\label{etainst}
    d\Big(\sum_{\gamma}\Omega(\gamma)\eta_{\gamma}^{\text{inst}}\Big)=\sum_{\gamma}\left( \frac{i\Omega(\gamma)}{2}V^{\text{inst}}_{\gamma}dZ_{\gamma}\wedge d\overline{Z}_{\gamma}+\frac{\Omega(\gamma)}{2\pi } d\theta_{\gamma}\wedge A_{\gamma}^{\text{inst}}\right) \,,
\end{equation}
where we used that compact normal convergence (\ref{convergenceproperty}) lets us interchange sums and derivatives. The formula (\ref{omega3int}) then follows by comparison with (\ref{invKF}).
\end{proof}

\begin{corollary}
Let $\pi_N:P\to N$ and $\Theta$ be as in Proposition \ref{linebundle}. Then the connection $\eta \in \Omega^1(P)$ given by
\begin{equation}\label{etadef}
    \eta:=\Theta +2\pi \pi^*_N\Big(\frac{i}{4}\pi^*_M( \overline{\partial} r^2-\partial r^2) + \sum_{\gamma}\Omega(\gamma)\eta_{\gamma}^{\text{inst}}-\iota_Vg_N\Big)
\end{equation} 
has curvature 
\begin{equation}
    F=2\pi(\omega_3 -d(\iota_Vg_N)) \,.
\end{equation}

\end{corollary}
\begin{proof}
The fact that $d\eta=\pi^*_NF$ follows from the equations (\ref{curvaturetheta}) and (\ref{omega3int}). 
\end{proof}

The following proposition characterizes the part of $\eta$ containing the ``instanton corrections".
\begin{proposition} The hyperholomorphic connection $\eta$ on $\pi_N:P\to N$ can be written as

\begin{equation}\label{eta}
    \eta= \Theta+2\pi\pi^*_N\Big(-\frac{1}{2}\pi^*_Mr^2\widetilde{\eta} +\eta^{\text{inst}}\Big)
\end{equation}
where $\widetilde{\eta}=\frac{1}{r^2}\iota_{J\xi}g_M$ and
\begin{equation}\label{etainstdef}
    \eta^{\text{inst}}:=\sum_{\gamma}\Omega(\gamma)\eta_{\gamma}^{\text{inst}}-\iota_V(g_N-\pi^*_Mg_{M}) \,.
\end{equation}
Furthermore, if $\Omega=0$ then $\eta^\text{inst}=0$.

\end{proposition} 
\begin{proof}
We first show that $\Omega=0$ implies $\eta^{\text{inst}}=0$. In a local Darboux frame $(\widetilde{\gamma}_i,\gamma^i)$ of $\Lambda$ with $\text{Supp}(\Omega)\subset \text{span} \{\gamma^i\}$, we can write (recall Corollary \ref{HKmetriccoord})
\begin{equation}\label{etainstcoord}
    \eta^{\text{inst}}=\sum_{\gamma}\Omega(\gamma)\eta_{\gamma}^{\text{inst}}-\iota_V\Big( \sum_{\gamma}\Omega(\gamma)V^{\text{inst}}_{\gamma}|dZ_{\gamma}|^2+\frac{1}{4\pi^2}Y_iM^{ij} \overline{Y}_j\Big) \,.
\end{equation}
If $\Omega=0$, it follows that $Y_i=W_i+W_i^{\text{inst}}=W_i$, and $M_{ij}=N_{ij}+N_{ij}^{\text{inst}}=N_{ij}$, so $\iota_VW_i=0$ implies that
\begin{equation}
    \eta^{\text{inst}}=-\iota_V\Big(\frac{1}{4\pi^2}W_iN^{ij}\overline{W_j}\Big)=0 \,.
\end{equation}

To show the remaining identity, we notice that from (\ref{etadef}) and (\ref{etainstdef}) we can write

\begin{equation}
    \eta= \Theta +2\pi\pi^*_N\Big(\frac{i}{4}\pi^*_M( \overline{\partial} r^2-\partial r^2)  -\iota_V\pi^*_Mg_M +\eta^{\text{inst}}\Big)=\Theta +2\pi\pi^*_N\Big(\pi^*_M\Big(\frac{i}{4}( \overline{\partial} r^2-\partial r^2)  -\iota_{J\xi}g_M\Big) +\eta^{\text{inst}}\Big)\,.
\end{equation}

Furthermore, recalling (\ref{contform}):
\begin{equation}
    r^2\widetilde{\eta}=\iota_{J\xi}g_M=\frac{r^2}{2}d^c\log(r^2)=\frac{i}{2}(\overline{\partial}r^2-\partial r^2)
\end{equation}
we obtain (\ref{eta}).
\end{proof}
\end{subsection}
\subsection{HK-QK correspondence}\label{HK-QK}

From the data of an integral CASK manifold $(M,g_M,\omega_M,\nabla,\xi,\Lambda)$ with a compatible $\Omega$, we have obtained a pseudo-HK manifold $(N,g_N,I_1,I_2,I_3)$ with an infinitesimal rotating circle action $V$ and a hyperholomorphic circle bundle $(\pi_{N}:P\to N,\eta)$. We now wish to apply \cite[Theorem 2]{QKPSK} to obtain a pseudo-QK manifold. In order to match their conventions, we will need to consider $2\pi g_N$ instead of $g_N$, and take $X=2V$, so that $X$ satisfies $\mathcal{L}_XI_1=-2I_2$, $\mathcal{L}_XI_2=2I_1$, $\mathcal{L}_XI_3=0$, $\mathcal{L}_Xg_N=0$. On the other hand, in order to directly apply \cite[Theorem 2]{QKPSK}, we will need to define several other pieces of data associated to $(N,2\pi g_N,I_1,I_2,I_3)$ and $(\pi_N:P\to N,\eta)$. This are given in Lemmas \ref{ff1} and \ref{Nrestriction}, and in Definition \ref{thetadefs} below.

\begin{lemma} \label{ff1} Let $c\in \mathbb{R}$ and let $f,f_1\in C^{\infty}(N)$ be defined by 

\begin{equation}\label{defff1}
    f:= 2\pi r^2-c + 4\pi \sum_{\gamma}\Omega(\gamma)\iota_V\eta_{\gamma}^{\text{inst}}, \;\;\;\;\;\; f_1:=-2\pi r^2 -c +4\pi \iota_V\eta^{\text{inst}}\,.
\end{equation}
Then $f$ and $f_1$ satisfy
\begin{equation}\label{condff1}
    df=-\iota_X(2\pi \omega_3)=-4\pi \iota_V\omega_3, \;\;\;\;\;\; f_1=f-\frac{1}{2}(2\pi g_N(X,X))=f-4\pi g_N(V,V)\,.
\end{equation}
Furthermore, if 
\begin{equation}\label{ff1inst}
    f^{\text{inst}}:= 4\pi \sum_{\gamma}\Omega(\gamma)\iota_V\eta_{\gamma}^{\text{inst}}\;\;\;\;\;\; f_1^{\text{inst}}:=4\pi \iota_V\eta^{\text{inst}}\,,
\end{equation}
then $\Omega=0$ implies $f^{\text{inst}}=f_1^{\text{inst}}=0$.
\end{lemma}
\begin{proof}
We start by proving that $df=-4\pi \iota_V\omega_3$. Using conical holomorphic special coordinates $Z^i$ we have the local expression 
    \begin{equation}\label{coordV}
        V= iZ^i\partial_{Z^i}-i\overline{Z}^i\partial_{\overline{Z}^i} \;\;\;\;\; r^2=N_{ij}Z^i\overline{Z}^j=\text{Im}(\tau_{ij})Z^i\overline{Z}^j\,.
    \end{equation}
On the other hand we have the identity $ Z^i\frac{\partial \tau_{ij}}{\partial Z^k}=0$ as a consequence of the CASK condition. It is then easy to check that 
\begin{equation}
    \frac{i}{2}\iota_V(\partial \overline{\partial}r^2)=-\frac{1}{2}d(r^2)\,.
\end{equation}
From the same arguments of Proposition \ref{rotatingprop}, it follows that $\mathcal{L}_{V}\eta_{\gamma}^{\text{inst}}=0$. We then have

\begin{equation}
        df=2\pi d(r^2) +4\pi \sum_{\gamma}\Omega(\gamma)d(\iota_V\eta_{\gamma}^{\text{inst}})=-4\pi\Big(\frac{i}{2}\iota_V(\partial \overline{\partial}r^2)+\iota_Vd\Big(\sum_{\gamma}\Omega(\gamma)\eta_{\gamma}^{\text{inst}}\Big)\Big)=-4\pi \iota_V\omega_3 \,.
\end{equation}

The remaining equation in (\ref{condff1}) for $f_1$ follows easily from $\pi^*_Mg_M(V,V)=g_M(J\xi,J\xi)=g_M(\xi,\xi)=r^2$ and the definition of $\eta^{\text{inst}}$ in (\ref{etainstdef}).\\

Finally, if $\Omega=0$ we clearly have $f^{\text{inst}}=0$, while $f_1^{\text{inst}}=0$ follows from the fact that $\Omega=0$ implies $g_N(V,V)=g^{\text{sf}}(V,V)=\pi^*_Mg_M(V,V)$.
\end{proof}

\begin{lemma}\label{Nrestriction}
Assume that the flow of $\xi$ induces a free-action on $M$ of the monoid $\mathbb{R}_{\geq 1}$. Then the open subset of $N$ defined by

\begin{equation}\label{N'}
    N':=\{ p \in N \;\; | \;\; f(p)\neq 0, \;\; f_1(p)\neq 0, \;\; g_N(X_p,X_p)\neq 0 \} \subset N
\end{equation} is not empty. In particular, the open subset $N_+':=\{ f>0 ,\;\;  f_1<0 \}\subset N'$ is not empty.
\end{lemma}
\begin{proof}
The proof is similar to Proposition \ref{signatureHK}. Indeed, fix $p\in N$ with $\pi_M(p)=q$, consider the ray $l_q:=\mathbb{R}_{\geq 1}\cdot q \subset M$, and pick some Darboux frame $(\widetilde{\gamma}_i,\gamma^i)$ of $\Lambda$ along the ray $l_q$ with $\text{Supp}(\Omega)\subset \text{span}\{\gamma^i\}$. By consider the $\nabla$-flat lift of $\xi$ to $N$, we obtain a $\mathbb{R}_{\geq 1}$-action on $\pi^{-1}_M(l_q)\subset N$ that rescales $Z_{\gamma}$, and leaves invariant $\theta_{\gamma}$. Writing $f^{\text{inst}}$, $f_1^{\text{inst}}$ with respect to $(\widetilde{\gamma}_i,\gamma^i)$ and using the asymptotics of the Bessel functions together with the convergence property, it is easy to see that given any $\epsilon>0$ the following holds for sufficiently big $t \in \mathbb{R}_{\geq 1}$
\begin{equation}
    |f^{\text{inst}}(t\cdot p)|<\epsilon, \;\;\;\;\;\;  |f^{\text{inst}}(t\cdot p)-f_1^{\text{inst}}(t\cdot p)|=\pi |g_N(X_{t \cdot p},X_{t \cdot p})-\pi^*_Mg_M(X_{t \cdot p},X_{t \cdot p})|<\epsilon
\end{equation}
where in the last inequality we used the fact that $g_N(X,X)-\pi^*_Mg_M(X,X)$ only contains instanton correction terms (see (\ref{HKmetric}) and (\ref{coordV})). \\

In particular, for sufficiently big $t$ we have
\begin{equation}
    \begin{split}
    \text{sign}(f(t\cdot p))=\text{sign}(2\pi t^2r^2(q)-c)=1,& \;\;\; \text{sign}(f_1(t\cdot p))=\text{sign}(-2\pi t^2r^2(q)-c)=-1\\
    \text{sign}(g_N(X(t\cdot p)&,X(t\cdot p)))=\text{sign}(t^2r^2(q))=1\\
    \end{split}
\end{equation} 
so $N'$ contains a point $t\cdot p$ where $f>0$, $f_1<0$ and $g_N(X,X)>0$.
\end{proof}

\begin{definition}\label{thetadefs}
     Let $N'\subset N$ be as in Lemma \ref{Nrestriction}. On the total space of  $(\pi_N:P|_{N'} \to N',\eta)$ we define the following objects:
    
    \begin{itemize}
        \item We  endow $P|_{N'}$ with the pseudo-Riemannian metric:

    \begin{equation}
        g_P:=\frac{2}{f_1}\eta^2 +2\pi \pi^*_Ng_N
    \end{equation}
    and the vector field
    \begin{equation}
        X_1^P:=\widetilde{X}+f_1\partial_{\sigma}\,,
    \end{equation}
    where $\widetilde{X}$ denotes the horizontal lift of $X$ and $\partial_{\sigma}$ the generator of the $S^1$-principal action.
    \item Finally, we define the following $1$-forms on $P|_{N'}$:
\begin{equation}\label{deftheta}
        \theta_0^P:=-\frac{1}{2}\pi^*_Ndf,\;\;\;\; 
        \theta_1^P:=\eta+\frac{1}{2}\pi^*_N\iota_X(2\pi g_N),\;\;\;\;
        \theta_2^P:=\frac{1}{2}\pi^*_N\iota_X(2\pi\omega_2),\;\;\;\;
        \theta_3^P:=-\frac{1}{2}\pi^*_N\iota_X(2\pi \omega_1)\,.
\end{equation}
    \end{itemize}
    
\end{definition}

We can now state the main theorem of this section:
\begin{theorem}\label{theorem2} Let  $(M,g_M,\omega_M,\nabla,\xi,\Lambda)$ be a connected integral CASK manifold with a compatible mutually local $\Omega$, and assume that the flow of $\xi$ generates a free-action of the monoid $\mathbb{R}_{\geq 1}$. Furthermore, let $(N,g_N,I_1,I_2,I_3)$ be the associated pseudo-HK manifold; $(P\to N,\eta)$ the associated hyperholomorphic circle bundle; and $\theta_i^P$, $g_P$, $X_1^P$, $f$ and $N'\subset N$ as before. If $\overline{N}\subset P|_{N'}$ is any submanifold transversal to $X_1^P$, then 

\begin{equation}\label{QKmetricglobal}
    g_{\overline{N}}:=-\frac{1}{f}\Big(g_P - \frac{2}{f}\sum_{i=0}^3 (\theta_i^P)^2\Big)\Big|_{\overline{N}}
\end{equation}
is a pseudo-QK metric on $\overline{N}$. Furthermore, if $\overline{N}$ is picked to also satisfy $\overline{N}_+:=\overline{N}\cap N_+'\neq \emptyset$, then $g_{\overline{N}}$ is positive definite on $\overline{N}_+$.
\end{theorem}

\begin{proof}
We have set everything up so that we can apply the explicit formulas of the HK-QK correspondence in \cite[Theorem 2]{QKPSK} to the pseudo-HK manifold $(N',2\pi g_N,I_1,I_2,I_3)$ with the vector field $X=2V$ and the hyperholomorphic circle bundle $(P|_{N'}\to N',\eta)$ (notice that we restricted to $N'$ of Lemma \ref{Nrestriction}). \cite[Theorem 2]{QKPSK} then guarantees that 
\begin{equation}
    g':=\frac{1}{2|f|}\Big(g_P - \frac{2}{f}\sum_{i=0}^3 (\theta_i^P)^2\Big)\Big|_{\overline{N}}
\end{equation}
is a pseudo-QK metric on $\overline{N}$. Hence, $g_{\overline{N}}=-2\text{sign}(f)g'$ is also a pseudo-QK metric on $\overline{N}$. \\

By Proposition \ref{signatureHK} we know that $g_N$ has signature $(4,4n)$, so by  \cite[Corollary 1]{Conification} we conclude from the conditions $f>0$, $f_1<0$  that $g'=-\frac{\text{sign}(f)}{2}g_{\overline{N}}=-\frac{1}{2}g_{\overline{N}}$ is negative definite on $\overline{N}_+$, and hence $g_{\overline{N}}$ is positive definite on $\overline{N}_+$.
\end{proof}

\begin{remark} Let us make a few comments on the assumption in Theorem \ref{theorem2} that the flow of $\xi$ generates a free-action of the monoid $\mathbb{R}_{\geq 1}$:

\begin{itemize}
    \item On one hand, this assumption is satisfied if one starts with a CASK domain (see Definition \ref{CASKdomdef}), since in this case $M\subset \mathbb{C}^{n+1}-\{0\}$ is a $\mathbb{C}^{\times}$-invariant domain, and $\xi$ corresponds to the infinitesimal action of the rescaling action $\mathbb{R}_{>0}\subset \mathbb{C}^{\times}$ (and hence, in particular, to the monoid action $\mathbb{R}_{\geq 1}\subset \mathbb{R}_{>0}$). The case of the CASK domain will be important below, since it gives rise to an instanton deformation of the $1$-loop corrected Ferrara-Sabharwal metric, to be discussed below in Section \ref{QKdefFS}.
    \item The assumption on $\xi$ is not always satisfied from the conditions of Definition \ref{defCASK}. Indeed, it fails if one starts from the previous case of a CASK domain $M$, and a point from a $\mathbb{R}_{\geq 1}$-orbit is removed. It also fails if we remove a whole neighborhood of $\infty$ of the $\mathbb{C}^{\times}$-invariant domain $M$. In both cases one obtains a CASK manifold according to Definition \ref{defCASK} that does not satisfy the above condition on $\xi$.
    \item On the other hand, this assumption is needed for two key points. First to guarantee that the open subset $N'\subset N$ where we can apply the HK-QK correspondence is non-empty (Lemma \ref{Nrestriction} and \cite[Theorem 2]{QKPSK}). And second, to guarantee that the HK metric $g_{N}$ has the appropriate signature (Proposition \ref{signatureHK}) in order for $g_{\overline{N}}$ to be positive definite on $\overline{N}_{+}$ (\cite[Corollary 1]{Conification}).
\end{itemize}

\end{remark}

\end{section}
\begin{section}{Instanton deformations of the 1-loop corrected Ferrara-Sabharwal metric}\label{QKdefFS}

In this section we wish to compute a coordinate expression for the QK metric $g_{\overline{N}}$ of Theorem \ref{theorem2} in the case of an integral CASK domain with a compatible mutually local $\Omega$. This will allow us to see $g_{\overline{N}}$ as a deformation of the $1$-loop corrected Ferrara-Sabharwal metric $g_{\text{FS}}^c$ (see Theorem \ref{propcoordQK}). Furthermore, we specify in Proposition \ref{signatureQK} when the QK metric is positive definite, and we discuss the fate of certain Peccei-Quinn symmetries in Corollary \ref{PQsymm}.

\begin{definition}\label{CASKdomdef}\cite{ACD,CDS}
A CASK domain is tuple $(M,\mathfrak{F})$ where:
\begin{itemize}
    \item  $M\subset \mathbb{C}^{n+1}-\{0\}$ is a $\mathbb{C}^{\times}$-invariant domain. We denote the canonical holomorphic coordinates of $\mathbb{C}^{n+1}$ by $z^i$, $i=0,1,...,n$.
    \item $\mathfrak{F}:M \to \mathbb{C}$ is a holomorphic function, homogeneous of degree $2$ with respect to the natural $\mathbb{C}^{\times}$-action on $M$. 
    \item The matrix 
    \begin{equation}
        N_{ij}=\text{Im}\Big(\frac{\partial^2 \mathfrak{F}}{\partial z^i \partial z^j} \Big)
    \end{equation}
    has signature $(1,n)$ and $N_{ij}z^i\overline{z}^j>0$ for all $z \in M$.
\end{itemize}

\end{definition}

To any CASK domain $(M,\mathfrak{F})$ we can associate a CASK manifold
$(M,g_M,\omega_M,\nabla,\xi)$ \cite{ACD} where $z^i$ and $w_i=\frac{\partial{\mathfrak{F}}}{\partial z^i}(z^i)$
    form a global system of conjugate conical holomorphic special coordinates. If $x^i:=\text{Re}(z^i)$ and $y_i:=-\text{Re}(w_i)$, then $\nabla$ is defined such that $dx^i$ and $dy_i$ are flat. Furthermore
    \begin{equation}
        g_M=N_{ij}dz^id\overline{z}^j, \;\;\;\;\;\; \omega_M=\frac{i}{2}N_{ij}dz^i\wedge d\overline{z}^j=dx^i\wedge dy_i, \;\;\;\;\;\; \xi=z^i\partial_{z^i}+\overline{z}^i\partial_{\overline{z}^i}\,,
    \end{equation}
 and if we define $\Lambda \to M$ by $\Lambda:=\text{Span}_{\mathbb{Z}}\{\partial_{x^i},\partial_{y_i}\}$, then  $(M,g_M,\omega_M,\nabla,\xi,\Lambda)$ is an integral CASK manifold. 

\begin{definition} A triple $(M,\mathfrak{F},\Lambda)$ where  $(M,\mathfrak{F})$ is a CASK domain and $\Lambda\to M$ is the canonical integral lattice from above above will be called an integral CASK domain. 
\end{definition}

Now consider an integral CASK domain $(M,\mathfrak{F},\Lambda)$, and let $M_{\infty}\subset M$ be an open subset invariant under the $S^1\subset \mathbb{C}^{\times}$-action and under the monoid action $\mathbb{R}_{\geq 1}\subset \mathbb{C}^{\times}$. Let $(M_{\infty},g_M,\omega_M,\nabla,\xi,\Lambda)$ be the corresponding integral CASK manifold together with a compatible $\Omega$ such that $\text{Supp}(\Omega)\subset \text{span}_{\mathbb{Z}}\{\partial_{y_i}\}$. 

\begin{remark}
The main reason to possibly restrict to $M_{\infty}$ is that in general it seems easier to find compatible, non-trivial mutually local variations of BPS structures on $M_{\infty}$ than on $M$ (see the example of Section \ref{example}). Furthermore, notice that if $(\overline{M},g_{\overline{M}},\omega_{\overline{M}})$ is the associated PSK manifold to $(M,\mathfrak{F})$ and $\pi_{\overline{M}}:M\to \overline{M}$ the projection, then $\pi_{\overline{M}}(M_{\infty})=\overline{M}$.
\end{remark}
We wish to compute the pseudo-QK metric $g_{\overline{N}}$  from Theorem \ref{theorem2} associated to $(M_{\infty},g_M,\omega_M,\nabla,\xi,\Lambda)$ and $\Omega$, when we take

\begin{equation}\label{QKN}
    \overline{N}:=\{\text{Arg}(z^0)=0\}\subset P|_{N'}\,.
\end{equation} 
This submanifold $\overline{N}$ is transverse to $X_1^P=\widetilde{X}+f_1\partial_{\sigma}$ since $\widetilde{X}|_{\overline{N}} =2\widetilde{V}|_{\overline{N}}\not\in T\overline{N}$ and $\partial_{\sigma}|_{\overline{N}} \in T\overline{N}$.
On \cite[Theorem 5]{QKPSK} it was shown that the case with $\Omega=0$ gives the 1-loop corrected Ferrara-Sabharwal metric $g_{\text{FS}}^c$, so the case where $\Omega\neq 0$ will give a deformation of $g_{\text{FS}}^c$.\\

On $M$ we have the local coordinates
    \begin{equation}
        r=\sqrt{g_M(\xi,\xi)}=\sqrt{N_{ij}z^i\overline{z}^j}>0, \;\;\;\; \phi:=\text{Arg}(z^0), \;\;\;\; X^i:=z^i/z^0 \;\; \text{for} \;\; i=1,2,...,n\,.
    \end{equation}
    As in \cite{QKPSK}, we will replace $r$ with $\rho=2\pi r^2 -c$, the later coordinate representing the dilaton coordinate for $g_{\text{FS}}^c$. We then have on $\overline{N}$ the local coordinates $\rho$, $X^i$ (for $i=1,2,..n$), $\widetilde{\theta}_i:=\theta_{\partial_{x^i}}$, $\theta^i:=\theta_{\partial_{y_i}}$ and $\sigma$. The coordinates $\rho$ and $X^i$ are global on $\overline{N}$, while the others satisfy that if $\delta \in \Lambda^*$ and $\theta \to \theta+2\pi \delta$, then $\sigma \to \sigma +\frac{1}{2} \langle \delta,\theta \rangle$ (see Proposition \ref{linebundle}). We will furthermore use the following ``normalized" central charge $X_{\gamma}:=Z_{\gamma}/z^0$. In particular, for $\gamma \in \text{Supp}(\Omega)$ with $\gamma=n_i(\gamma)\partial_{y_i}$, we have $X_{\gamma}=n_i(\gamma)X^i$ where $X^0:=1$ and $X^i$ for $i=1,2,...,n$ are as before.

\begin{theorem}\label{propcoordQK} Consider an integral CASK domain $(M,\mathfrak{F},\Lambda)$ and let $(\overline{M},g_{\overline{M}})$ be the associated PSK manifold. Furthermore, let $(M_{\infty},g_M,\omega_M,\nabla,\xi,\Lambda)$ be as before with a compatible mutually local $\Omega$ satisfying $\text{Supp}(\Omega)\subset \text{span}_{\mathbb{Z}}\{\partial_{y_i}\}$. Taking $\overline{N}$ as in (\ref{QKN}), the expression for the QK metric $g_{\overline{N}}$ of Theorem \ref{theorem2} in the coordinates $\rho$, $X^i$, $\widetilde{\theta}_i$, $\theta^i$ and $\sigma$ takes the form: 

\begin{equation}\label{coordQKmetric2}
    \begin{split}
        g_{\overline{N}}=& \frac{\rho+c}{\rho+f^{\text{inst}}}\Big(g_{\overline{M}}-e^{\mathcal{K}}\sum_{\gamma}\Omega(\gamma)V_{\gamma}^{\text{inst}}\Big|dX_{\gamma} + X_{\gamma}\Big(\frac{d\rho}{2(\rho+c)}+\frac{d\mathcal{K}}{2}\Big)\Big|^2\Big)\\
        &+\frac{1}{2(\rho+f^{\text{inst}})^2}\Big(\frac{\rho +2c -f^{\text{inst}}}{2(\rho+c)}d\rho^2+2d\rho df^{\text{inst}}|_{\overline{N}}+(df^{\text{inst}})^2|_{\overline{N}}\Big)\\
        &+\frac{4(\rho+c+f_-^{\text{inst}})}{(\rho+f^{\text{inst}})^2(\rho+2c-f_1^{\text{inst}})}\Big(d\sigma -\frac{1}{4\pi}\langle \theta,d\theta \rangle -\frac{c}{4}d^c\mathcal{K}+\eta_+^{\text{inst}}|_{\overline{N}}+\frac{f_+^{\text{inst}}-c}{\rho+c+f_-^{\text{inst}}}\eta_-^{\text{inst}}|_{\overline{N}}\Big)^2\\
        &-\frac{1}{2\pi(\rho+f^{\text{inst}})}(W_i+W_i^{\text{inst}}|_{\overline{N}})(N+N^{\text{inst}})^{ij}(\overline{W}_j+\overline{W}_j^{\text{inst}}|_{\overline{N}}) \\
        &+\frac{(\rho+c)e^{\mathcal{K}}}{\pi(\rho+f^{\text{inst}})^2}\Big|X^i(W_i+W_i^{\text{inst}}|_{\overline{N}})+2\pi i\sum_{\gamma}\Omega(\gamma) A_{\gamma}^{\text{inst}}(V)\Big(dX_{\gamma} + X_{\gamma}\Big(\frac{d\rho}{2(\rho+c)}+\frac{d\mathcal{K}}{2}\Big)\Big)\Big|^2\\
        &
        +\frac{\rho+c+f_-^{\text{inst}}}{\rho+f^{\text{inst}}}\Big(\frac{d^c\mathcal{K}}{2}+\frac{2}{\rho+c+f_-^{\text{inst}}}\eta_-^{\text{inst}}|_{\overline{N}}\Big)^2-\frac{\rho+c}{\rho+f^{\text{inst}}}\Big(\frac{d^c\mathcal{K}}{2}\Big)^2
    \end{split}
\end{equation}
where $\mathcal{K}=-\log(N_{ij}X^i\overline{X}^j)$ is a K\"{a}hler potential for $g_{\overline{M}}$, $\eta^{\text{inst}}_{\pm}$ are given by
\begin{equation}
    \eta_{\pm}^{\text{inst}}:=\frac{1}{2}\Big(\Big(2\pi \eta^{\text{inst}}-\frac{f_1^{\text{inst}}}{2}\widetilde{\eta}\Big)\pm \Big(2\pi \sum_{\gamma}\Omega(\gamma)\eta_{\gamma}^{\text{inst}}-\frac{f^{\text{inst}}}{2}\widetilde{\eta}\Big)\Big)\,,
\end{equation}
and 
\begin{equation}
    f_{\pm}^{\text{inst}}:=(f^{\text{inst}}\pm f_1^{\text{inst}})/2 \,.
\end{equation}
Furthermore, the forms $1$-forms $W_i^{\text{inst}}|_{\overline{N}}$ and  $\eta^{\text{inst}}_{\pm}|_{\overline{N}}$ do not have $d\rho$ components.

\end{theorem}

\begin{proof}
The proof is given in appendix \ref{appendixB}. Furthermore, coordinate expressions for  $df^{\text{inst}}|_{\overline{N}}$ and $W_i^{\text{inst}}|_{\overline{N}}$ are given in (\ref{dfinst}) and (\ref{winst}); while coordinate expressions for $\eta^{\text{inst}}_{\pm}|_{\overline{N}}$ can be found using (\ref{eta1inst}). 
\end{proof}

\begin{remark} By setting $\Omega=0$ in (\ref{coordQKmetric2}) all the instanton terms vanish, and we recover $g_{\text{FS}}^c$:
\begin{equation}\label{FSmetric}
    \begin{split}
    g_{\text{FS}}^c=&\frac{\rho+c}{\rho}g_{\overline{M}}+\frac{\rho +2c }{4(\rho+c)\rho^2}d\rho^2+\frac{4(\rho+c)}{\rho^2(\rho+2c)}\Big(d\sigma -\frac{1}{4\pi}\langle \theta,d\theta \rangle -\frac{c}{4}d^c\mathcal{K}\Big)^2\\
    &-\frac{1}{2\pi\rho}W_i\Big(N^{ij}-\frac{2(\rho+c)e^{\mathcal{K}}}{\rho}X^i\overline{X}^j\Big)\overline{W}_j\,.\\
    \end{split}
\end{equation}
Indeed, by scaling $\rho \to \rho/\pi$, $c\to c/\pi$, $\sigma \to -\sigma/4\pi$, and $\theta^i\to -\theta^i$, together with \cite[Equation 4.12]{QKPSK} and the fact that $N_{ij}$ from \cite{QKPSK} differs by a factor of $2$ with respect to our $N_{ij}$; we obtain \cite[Equation 4.11]{QKPSK}.
\end{remark}

\begin{proposition}\label{signatureQK} The subset $\overline{N}_+=\{f>0,\;\;\; f_1<0\} \subset \overline{N}$ is non-empty, and $g_{\overline{N}}$ is positive-definite on $\overline{N}_+$. Furthermore, there is a non-empty neighborhood of $\overline{N}_{\infty}\subset \overline{N}_{+}$ of $\rho =\infty$ where $g_{\text{FS}}^c$ is defined and positive definite. In particular, $g_{\overline{N}}$ gives a deformation of $g_{\text{FS}}^c$ on $\overline{N}_{\infty}$.
\end{proposition}
\begin{proof}
Fix a point $(\rho,X^i,\theta^i,\widetilde{\theta}_i,\sigma)\in \overline{N}$. The key observation is that with respect to our coordinates $f^{\text{inst}}, f_1^{\text{inst}}\to 0$ as $\rho \to \infty$, due to the exponential decay of the Bessel functions as $\rho \to \infty$ and the convergence property (\ref{convergenceproperty}). This follows from the dependence of the Bessel functions on $|Z_{\gamma}|$ and the identity

\begin{equation}
    |Z_{\gamma}|=|z^0||X_{\gamma}|=re^{\mathcal{K}(X)/2}|X_{\gamma}|=\sqrt{\frac{\rho+c}{2\pi}}e^{\mathcal{K}(X)/2}|X_{\gamma}|, \;\;\;\;\; \gamma \in \text{Supp}(\Omega)\,.
\end{equation}Hence, for sufficiently big $\rho$ we have $f=\rho + f^{\text{inst}}>0$ and $f_1=-\rho -2c + f_1^{\text{inst}}<0$, and hence $(\rho,X^i,\theta^i,\widetilde{\theta}_i,\sigma)\in \overline{N}_+$. By Theorem \ref{theorem2} it then follows that $g_{\overline{N}}$ is positive definite on $\overline{N}_{+}$.\\

On the other hand, \cite[Theorem 5]{QKPSK} shows that $g_{\text{FS}}^c$ is positive definite and defined on $\{\rho > \text{max}\{0,-2c\}\}\subset \overline{N}$. Hence, taking $\overline{N}_{\infty}=\overline{N}_{+}\cap \{\rho > \text{max}\{0,-2c\}\}$ gives the required neighborhood of $\rho=\infty$.
\end{proof}

We finish this section with a few words about the Heisenberg group of isometries for certain lifts of $g_{\text{FS}}^c$ and $g_{\overline{N}}$.  Consider the trivial bundle of Heisenberg groups $\overline{M}\times \mathbb{R}_{\rho>0}\times \text{Heis}_{2n+3}(\mathbb{R})$, where we identify $\text{Heis}_{2n+3}(\mathbb{R})\cong \mathbb{R}^{2(n+1)+1}$ and suggestively denote a point of $\mathbb{R}^{2(n+1)+1}$ by $(\widetilde{\theta}_i,\theta^i,\sigma)=(\theta,\sigma)$. The group structure is given by

\begin{equation}
    (\theta,\sigma)\cdot (\theta',\sigma')=(\theta+\theta',\sigma + \sigma' +\frac{1}{4\pi}\langle \theta,\theta' \rangle)\,.
\end{equation}

Both metrics $g_{\overline{N}}$ and $g_{\text{FS}}^c$ lift to an open subset of $\overline{M}\times \mathbb{R}_{\rho>0}\times \text{Heis}_{2n+3}(\mathbb{R})$. Indeed, we just take the same formulas (\ref{coordQKmetric2}) and (\ref{FSmetric}) as before, where now $(\widetilde{\theta}_i,\theta^i,\sigma)$ are global coordinates of $\mathbb{R}^{2(n+1)+1}$. Then we can explicitly check that $\text{Heis}_{2n+3}(\mathbb{R})$ acts by isometries on $g_{\text{FS}}^c$, while for $g_{\overline{N}}$ we have the following:

\begin{corollary} \label{PQsymm} Consider the previous lifts of $g_{\overline{N}}$ to an open subset of $\overline{M}\times \mathbb{R}_{\rho>0}\times \text{Heis}_{2n+3}(\mathbb{R})$. For $\gamma \in \text{Supp}(\Omega)$ we denote $\gamma=n_i(\gamma)\partial_{y_i}$, and we define $d_i:=\text{gcd}(\{n_i(\gamma)\}_{\gamma\in \text{Supp}(\Omega)})$ for $i=0,1...,n$ such that $\{n_i(\gamma)\}_{\gamma \in \text{Supp}(\Omega)}\neq \{0\} $. If $\Omega\neq 0$, then the following proper subgroup  of $\text{Heis}_{2n+3}(\mathbb{R})$ acts by isometries on the lift of $g_{\overline{N}}$:
\begin{equation}
    \{ (\widetilde{\eta}_i,\eta^i,\kappa)\in \text{Heis}_{2n+3}(\mathbb{R}) \;\; | \;\; \eta^i\in \frac{2\pi \mathbb{Z}}{d_i} \;\; \text{for} \;\; i=0,1,...,n \;\; \text{such that} \;\; \{n_i(\gamma)\}_{\gamma\in \text{Supp}(\Omega)}\neq \{0\} \}\,.
\end{equation}

\end{corollary}

\begin{proof}
This follows from the explicit formula (\ref{coordQKmetric2}), together with the fact that in all the instanton correction terms the  coordinates $(\widetilde{\theta}_i,\theta^i)$ only enter in terms of functions of the form $e^{im\theta_{\gamma}}=e^{imn_j(\gamma)\theta^j}$ with $m\in \mathbb{Z}$ and $\gamma \in \text{Supp}(\Omega)$, breaking the allowed shifts of $\theta^j$ to discrete shifts. In particular, if $j$ is such that $\{n_j(\gamma)\}_{\gamma\in \text{Supp}(\Omega)}\neq 0$, then a shift in $\theta^j$ must preserve $e^{in_j(\gamma)\theta^j}$ for all $\gamma \in \text{Supp}(\Omega)$, which implies that the allowed shifts of $\theta^j$ must lie in $2\pi \mathbb{Z}/d_j$.
\end{proof}

\begin{remark}In the physics parlance, the previous corollary says that including ``mutually local instanton corrections" break the allowed shifts of the ``electric" angles $\theta^j$ to discrete shifts. Non-mutually local instanton corrections are expected to further break down the allowed shifts of the ``magnetic" angles $\widetilde{\theta}_j$ to discrete shifts.  
\end{remark}

\begin{subsection}{Comments on the metric}
\label{commentssec}
We make some comments on our expression (\ref{coordQKmetric2}) compared to the one obtained in $(3.6)$ of \cite{HMmetric}.

\begin{itemize}
    \item The function $f=\rho + f^{\text{inst}}$ that appears as the Hamiltonian for the infinitesimal rotating action of $(N,g_N,\omega_1,\omega_2,\omega_3)$ with respect $\omega_3$ can be thought as a ``quantum corrected" dilaton coordinate. Indeed, in  \cite{HMreview1,HMreview2,HMmetric} the dilaton coordinate is built out of what they call a ``contact potential" of the QK twistor space, and this contact potential can be related to a Hamiltonian for the infinitesimal rotating action (see for example the end remarks of \cite{AMNP}).
    
    In the twistor approach of $\cite{HMmetric}$, it seems that the natural coordinate to consider is the quantum corrected dilaton $f$ instead of $\rho$. The fact that they use $f$ as a coordinate is what gives rise to their implicitly defined $\mathcal{R}$-function in their expression for the instanton corrected metric (see their equations $(3.6)$ and $(3.7)$).  
    \item We have an ``$S^1$-bundle term" or ``NS-axion bundle term", given by
    \begin{equation*}
       4\frac{\rho+c+f_-^{\text{inst}}}{(\rho+f^{\text{inst}})^2(\rho+2c-f_1^{\text{inst}})}\Big(d\sigma -\frac{1}{4\pi}\langle \theta,d\theta \rangle -\frac{c}{4}d^c\mathcal{K}+\eta_+^{\text{inst}}|_{\overline{N}}+\frac{f_+^{\text{inst}}-c}{\rho+c+f_-^{\text{inst}}}\eta_-^{\text{inst}}|_{\overline{N}}\Big)^2
    \end{equation*}
    in (\ref{coordQKmetric2}). Since $d\sigma$ appears only in this term, $g_{\overline{N}}$ has an $S^1$-action by isometries (as expected from HK-QK correspondence). Furthermore, by using the expression (\ref{eta1inst}) we see that the connection form of the circle bundle does not have a $d\rho$ component. In \cite{HMmetric} an analogous conclusion is reached in terms of the quantum corrected dilaton direction $df$.
    \item In the case of $g_{\text{FS}}^c$ the dilaton coordinate $\rho$ is orthogonal to the rest of the coordinates. This is not the case when including the instanton corrections, due the non-trivial mixing of $d\rho$ with $dX^i$, $d\theta^i$ and $d\widetilde{\theta}_i$ in (\ref{coordQKmetric2}). However, by the previous comment there is no mixing between $d\rho$ and $d\sigma$, so these two directions remain orthogonal. As in the previous point, in \cite{HMmetric} an analogous conclusion is obtained in terms of $f=\rho + f^{\text{inst}}$.
    \item The symplectic invariance that needs to be checked in \cite{HMmetric} is automatic from our construction. Indeed, all the objects of Section \ref{instcorrsec} and \ref{hyperholsec} involved in the construction of $g_{\overline{N}}$ do not depend on the choice of any Darboux frame, so the (non-explicit) symplectic invariance of formula (\ref{coordQKmetric2}) follows automatically. 
\end{itemize}
\end{subsection}

\begin{subsection}{An example}\label{example}
Here we present a simple example of our previous constructions where the PSK manifold associated to the CASK manifold is the complex hyperbolic space $\mathbb{C}H^n$. The choice of variation of mutually local BPS structures that we will take will be due to its mathematical simplicity, rather than its physical significance. We remark, however, that in the case $n=0$ where the PSK manifold reduces to a point (i.e. $\mathbb{C}H^0=\{*\}$), the example will give an instance of an instanton deformation of the universal hypermultiplet metric near its infinite end of finite volume, obtained  after taking a quotient by a lattice in the Heisenberg group (see Corollary \ref{finitevolend} below).\\ 

We start with a CASK domain $(M,\mathfrak{F})$ given as follows: 

\begin{itemize}
    \item We take $M= \{(z^0,z^1,...,z^n)\in \mathbb{C}^{n+1} \;\; | \;\; |z^0|^2>\sum_{i=1}^n |z^i|^2 \}$.
    \item The holomorphic prepotential  is given by
    \begin{equation}
        \mathfrak{F}(z^0,...,z^n)=\frac{i}{2}((z^0)^2-\sum_{i=1}^n(z^i)^2)\,.
    \end{equation}
\end{itemize}
In particular, we obtain a CASK manifold $(M,g_M,\omega_M,\nabla,\xi)$ where $\nabla=D=d$ and:
\begin{equation}
    \omega_M=\frac{i}{2}(dz^0\wedge d\overline{z}^0 - \sum_{i=1}^n dz^i \wedge d\overline{z}^i),\;\;\;\;\;\;\;\;g_M=dz^0d\overline{z}^0 - \sum_{i=1}^n dz^id\overline{z}^i, \;\;\;\;\;\;\; \xi= z^i\partial_{z^i}+\overline{z}^i\partial_{\overline{z}^i}\,.
\end{equation}
A global conjugate system of conical special holomophic coordinates is given by
    \begin{equation}
        z^i,\;\;\;\;\; w_i=\tau_{ij}z^j\,,
    \end{equation}
    where $\tau_{ij}=\text{diag}(i,-i,...,-i)$.
     In the corresponding global conical affine special coordinates $(x^i,y_i)=(\text{Re}(z^i),-\text{Re}(w_i))$ we set $\Lambda=\text{Span}_{\mathbb{Z}}\{\partial_{x^i},\partial_{y_i}\}$, so that $(M,\mathfrak{F},\Lambda)$ is an integral CASK domain. \\

The associated PSK manifold $(\overline{M},g_{\overline{M}},\omega_{\overline{M}})$ is described as follows: 

\begin{itemize}
    \item We have $\overline{M}=M/\mathbb{C}^{\times}$. If $n=0$ then $\overline{M}$ is just a point; otherwise $\overline{M}=\{ (X^1,...,X^n)\in \mathbb{C}^n \;\; | \;\; \sum_{i=1}^n |X^i|^2<1 \}$ with projection $\pi_{\overline{M}}: M\to \overline{M}$ given by $\pi(z^0,z^1,...,z^n)=(z^1/z^0,...,z^n/z^0)$.
    \item Setting $X^i=z^i/z^0$ for $i=0,1,...,n$, the K\"{a}hler potential for $g_{\overline{M}}$ is given by
    \begin{equation}
        \mathcal{K}=-\log(X^i\text{Im}\Big(\frac{\partial^2 \mathfrak{F}}{\partial z^i \partial z^j}\Big)\overline{X}^j)=-\log ( 1-\sum_{i=1}^n|X^i|^2)\,,
    \end{equation}
    which gives the metric
    \begin{equation}
     g_{\overline{M}}=\frac{(1-\sum_{j=1}^n|X^j|^2)\sum_{i=1}^n|dX^i|^2 +|\sum_{i=1}^n\overline{X}^jdX^j|^2}{(1-\sum_{i=1}^n|X^i|^2)^2}\,.
    \end{equation}
    This identifies $(\overline{M},g_{\overline{M}},\omega_{\overline{M}})$ as the complex hyperbolic space $\mathbb{C}H^{n}$.
\end{itemize}

In the global Darboux frame $(\widetilde{\gamma}_i,\gamma^i):=(\partial_{x^i},\partial_{y_i})$ of $\Lambda$, we consider a mutually local variation of BPS structures with $\text{Supp}(\Omega)=\{\pm \gamma^0\}$ and $\Omega(\pm \gamma^0)=m$ for some $m\in \mathbb{Z}-\{0\}$. We remark that $Z_{\gamma^0}=z^0$ does not vanish on $M$, and hence $\Omega$ satisfies the support property. The rest of the properties required by a mutually local variation of BPS structures are trivial to check.

\begin{remark}
One can of course consider more complicated mutually local variations of BPS structures. The one chosen above is only to simplify the argument below regarding the non-degeneracy of the tensor $T$ (recall Definition \ref{defT}). 
\end{remark}

Notice that the tensor field

\begin{equation}
    T=\pi^*_M(g_M) + \sum_{\gamma}\Omega(\gamma) V_{\gamma}^{\text{inst}}\pi^*_M|dZ_{\gamma}|^2=\pi^*_M(g_M) + \Omega(\gamma^0) (V_{\gamma^0}^{\text{inst}}+V_{-\gamma^0}^{\text{inst}})\pi^*_M|dz^0|^2
\end{equation}
does not satisfy the compatibility condition for all points in $N=T^*M/\Lambda^*\cong M\times (S^1)^{2n+2}$ unless $\Omega=0$. However, due to the exponential decay of the Bessel functions in $V_{\gamma^0}^{\text{inst}}$ in the variable $|z^0|$, this can be fixed by restricting $M$ to the domain $M_K:=\{(z^0,...,z^n)\in M\;\;\; | \;\;\; |z^0|>K>0 \}$ for $K>0$ sufficiently big. $M_{K}$ carries a free $S^1$-action generated by $J\xi$, and a free action generated by $\xi$ of the monoid $\mathbb{R}_{\geq 1}\subset \mathbb{C}^{\times}$. In particular, $M_K\subset M$ is an open subset of the form $M_{\infty}$ considered in the first part of Section \ref{QKdefFS}, so we can apply Theorem \ref{propcoordQK} to obtain a QK metric $(\overline{N},g_{\overline{N}})$ associated to the PSK manifold $(\mathbb{C}H^n,g_{\overline{M}},\omega_{\overline{M}})$, and given by (\ref{coordQKmetric2}). We then obtain the following immediate corollary using Proposition \ref{signatureQK}:

\begin{corollary}\label{comhyperdef} There is a neighborhood $\overline{N}_{\infty}\subset \overline{N}_{+}$ of $\rho=\infty$, where both QK metrics $g_{\text{FS}}^c$ and $g_{\overline{N}}$ associated to $(\mathbb{C}H^{n},g_{\overline{M}})$ are defined and positive definite. In particular $g_{\overline{N}}$ gives a deformation of $g_{\text{FS}}^c$ on $\overline{N}_{\infty}$.

\end{corollary}

An interesting and simpler case is given by taking $n=0$, where $M=\mathbb{C}^{\times}$ and $\overline{M}=\mathbb{C}H^{0}=\{*\}$ reduces to a point. The 1-loop corrected Ferrara-Sabharval metric $g_{\text{FS}}^c$ associated to this case is known as the universal hypermultiplet. For $c\geq 0$ it is a complete QK metric \cite[Corollary 15]{CDS} on 

\begin{equation}
    \overline{N}_{\text{UH}}:= \frac{T^*\mathbb{C}^{\times}\times \mathbb{R_{\sigma}}}{\text{Heis}(\Lambda^*)}\Big|_{\mathbb{R}_{\rho>0}} \,,
\end{equation}
where $\text{Heis}(\Lambda^*)\to \mathbb{C}^{\times}$ and its fiber-wise action on $T^*\mathbb{C}^{\times}\times \mathbb{R}_{\sigma}\to \mathbb{C}^{\times}$ is as in the proof of Proposition \ref{linebundle}, and $\mathbb{R}_{\rho>0}= \{z^0 \in \mathbb{C}^{\times}\;\; | \;\; \rho= 2\pi |z^0|^2-c>0 , \;\; \text{Arg}(z^0)=0\}$. $\overline{N}_{\text{UH}}$ has only one non-compact direction given by $\rho$, and hence two infinite ends corresponding to $\rho \to \infty$ and $\rho \to 0$. In \cite[Theorem 4.6]{CRT} it is shown that for $\rho_0>0$, the subset $\{\rho>\rho_0\}\subset \overline{N}_{\text{UH}}$ has finite volume, while $\{0<\rho<\rho_0\}\subset \overline{N}_{\text{UH}}$ has infinite volume. Hence, we obtain:

\begin{corollary}\label{finitevolend} Let $(\overline{N}_{+},g_{\overline{N}})$ be the (positive definite) QK metric from the previous example, associated to the PSK manifold $\overline{M}=\{*\}$. Then $(\overline{N}_{+},g_{\overline{N}})$ gives a deformation of the universal hypermultiplet $(\overline{N}_{\text{UH}},g_{\text{FS}}^c)$ near its infinite end of finite volume. 

\end{corollary}

\end{subsection}
\end{section}

\appendix
\begin{section}{Coordinate computation of the instanton deformation of \texorpdfstring{$g_{\text{FS}}^c$}{TEXT}}\label{appendixB}

In this appendix we prove Theorem \ref{propcoordQK}.

\begin{proof}(of Theorem \ref{propcoordQK})  We start computing $-\frac{1}{f}\widetilde{g}_P$, where

\begin{equation}
    \widetilde{g}_P:=g_P - \frac{2}{f}\sum_{i=0}^3 (\theta_i^P)^2\,;
\end{equation} and then deal with the restriction to $\overline{N}$, giving $g_{\overline{N}}=-\frac{1}{f}\widetilde{g}_P|_{\overline{N}}$. To ease the notations, we will omit pullbacks by projections. Furthermore, we denote $Z^i:=Z_{\partial_{y_i}}=z^i$.\\ 

We will use the following expressions for $\theta_i^P$, which follow from their definition (\ref{deftheta}) together with the formulas for $\eta$, $f$  and $\omega_{\alpha}$ given in (\ref{etadef}), (\ref{defff1}) and (\ref{omega12}), respectively:

\begin{align*}
        \theta_0^P&=-\frac{1}{2}df=-2\pi rdr-\frac{1}{2}df^{\text{inst}} \\
        \theta_1^P&=\eta+\frac{1}{2}\iota_X(2\pi g_N)=\eta +2\pi \iota_Vg_N=\Theta+2\pi \Big(\frac{1}{2}r^2\widetilde{\eta} + \sum_{\gamma}\Omega(\gamma)\eta_{\gamma}^{\text{inst}}\Big)\\
        &=d\sigma+2\pi \Big(\frac{1}{2}r^2\widetilde{\eta} - \frac{1}{8\pi^2}\langle \theta, d\theta \rangle+ \sum_{\gamma}\Omega(\gamma)\eta_{\gamma}^{\text{inst}}\Big)\\
        \theta_2^P&=\frac{1}{2}\iota_X(2\pi\omega_2)=2\pi \iota_V\omega_2= \Big(\text{Re}(Z^iY_i)-\text{Im}\Big(\sum_{\gamma}\Omega(\gamma)2\pi A_{\gamma}^{\text{inst}}(V)dZ_{\gamma}\Big)\Big)\\
        &=\text{Re}\Big(Z^iW_i +Z^iW_i^{\text{inst}}+2\pi i\sum_{\gamma}\Omega(\gamma) A_{\gamma}^{\text{inst}}(V)dZ_{\gamma}\Big)\\
        \theta_3^P&=-\frac{1}{2}\iota_X(2\pi \omega_1)=-2\pi \iota_V \omega_1= \Big(\text{Im}(Z^iY_i)+\text{Re}\Big(\sum_{\gamma}\Omega(\gamma)2\pi A_{\gamma}^{\text{inst}}(V)dZ_{\gamma}\Big)\Big)\\
        &= \text{Im}\Big(Z^iW_i+Z^iW_i^{\text{inst}}+2\pi i\Big(\sum_{\gamma}\Omega(\gamma) A_{\gamma}^{\text{inst}}(V)dZ_{\gamma}\Big)\Big)\,. \numberthis
\end{align*}

We start by computing

\begin{equation}
    \frac{2}{f_1}\eta^2-\frac{2}{f}(\theta_1^P)^2\subset \widetilde{g}_P\,.
\end{equation}

In order to do this, we use the expression of $\eta$ in (\ref{eta}), of $f_1$ and $f$ in (\ref{defff1}), and of $f^{\text{inst}}$ and $f_1^{\text{inst}}$ in (\ref{ff1inst}) to get

\begin{equation}\label{linebundleterm1}
        \frac{2}{f_1}\eta^2-\frac{2}{f}(\theta_1^P)^2=\frac{2}{f_1}\Big(\Theta +\frac{c}{2}\widetilde{\eta}+\frac{f_1}{2}\widetilde{\eta}+ 2\pi \eta^{\text{inst}}-\frac{f_1^{\text{inst}}}{2}\widetilde{\eta}\Big)^2  
        -\frac{2}{f}\Big(\Theta +\frac{c}{2}\widetilde{\eta}+\frac{f}{2}\widetilde{\eta}+ 2\pi \sum_{\gamma}\Omega(\gamma)\eta_{\gamma}^{\text{inst}}-\frac{f^{\text{inst}}}{2}\widetilde{\eta}\Big)^2\,.
\end{equation}
We wish to rewrite this in such a way that $\Theta + \frac{c}{2}\widetilde{\eta}$ only appears in one term. We will use the notations
\begin{equation}\label{etapm}
    \begin{split}
    \eta_+^{\text{inst}}&:=\frac{1}{2}\Big(\Big(2\pi \eta^{\text{inst}}-\frac{f_1^{\text{inst}}}{2}\widetilde{\eta}\Big)+\Big(2\pi \sum_{\gamma}\Omega(\gamma)\eta_{\gamma}^{\text{inst}}-\frac{f^{\text{inst}}}{2}\widetilde{\eta}\Big)\Big)\\
    \eta_-^{\text{inst}}&:=\frac{1}{2}\Big(\Big(2\pi \eta^{\text{inst}}-\frac{f_1^{\text{inst}}}{2}\widetilde{\eta}\Big)-\Big(2\pi \sum_{\gamma}\Omega(\gamma)\eta_{\gamma}^{\text{inst}}-\frac{f^{\text{inst}}}{2}\widetilde{\eta}\Big)\Big)\\
    \end{split}
\end{equation}
so that we can rewrite (\ref{linebundleterm1}) as follows

\begin{equation}\label{linebundleterm2}
    \begin{split}
        \frac{2}{f_1}&\eta^2-\frac{2}{f}(\theta_1^P)^2
        =\frac{2}{f_1}\Big(\Theta +\frac{c}{2}\widetilde{\eta}+\frac{f_1}{2}\widetilde{\eta}+  \eta_+^{\text{inst}}+\eta_-^{\text{inst}}\Big)^2   -\frac{2}{f}\Big(\Theta +\frac{c}{2}\widetilde{\eta}+\frac{f}{2}\widetilde{\eta}+ \eta_+^{\text{inst}}-\eta_-^{\text{inst}}\Big)^2\\
        =&\Big(\frac{2}{f_1}-\frac{2}{f}\Big)\Big(\Theta +\frac{c}{2}\widetilde{\eta}+\eta_{+}^{\text{inst}}\Big)^2+2\Big(\frac{2}{f_1}+\frac{2}{f}\Big)\Big(\Theta +\frac{c}{2}\widetilde{\eta}+\eta_{+}^{\text{inst}}\Big)\eta_{-}^{\text{inst}}\\
        &+\frac{2}{f_1}\Big(\frac{f_1}{2}\widetilde{\eta}+\eta_{-}^{\text{inst}}\Big)^2 -\frac{2}{f}\Big(\frac{f}{2}\widetilde{\eta}-\eta_{-}^{\text{inst}}\Big)^2\,.
       \end{split}
\end{equation}
After completing the square for the terms with $\Theta$ in (\ref{linebundleterm2}) and organizing the remaining terms, the previous expression can be written as 

\begin{equation}\label{linebundleterm3}
        \frac{2}{f_1}\eta^2-\frac{2}{f}(\theta_1^P)^2
        =\Big(\frac{2}{f_1}-\frac{2}{f}\Big)\Big(\Theta +\frac{c}{2}\widetilde{\eta}+  \eta_+^{\text{inst}}+\frac{f+f_1}{f-f_1}\eta_-^{\text{inst}}\Big)^2+\Big(\frac{f_1}{2}-\frac{f}{2}\Big)\Big(\widetilde{\eta}+\Big(\frac{4}{f_1-f}\Big)\eta_-^{\text{inst}}\Big)^2\,.
\end{equation}

We now compute 
\begin{equation}
    2\pi g_N -\frac{2}{f}(\theta_0^P)^2\subset \widetilde{g}_P\,.
\end{equation}
We recall that we can write $g_{N}$  as (see Corollary \ref{HKmetriccoord}):

\begin{equation}
    g_N=\pi^*_Mg_M+\sum_{\gamma}\Omega(\gamma)V_{\gamma}^{\text{inst}}|dZ_{\gamma}|^2 + \frac{1}{4\pi^2}Y_iM^{ij}\overline{Y}_j\,,
\end{equation}
while using (\ref{metriccone}) and omitting in the notation pullbacks by projections we have

\begin{equation}
    2\pi g_N -\frac{2}{f}(\theta_0^P)^2=2\pi \Big( dr^2 +r^2(\widetilde{\eta}^2-g_{\overline{M}}) + \sum_{\gamma}\Omega(\gamma)V_{\gamma}^{\text{inst}}|dZ_{\gamma}|^2 + \frac{1}{4\pi^2}Y_iM^{ij}\overline{Y}_j\Big) -\frac{1}{2f}(df)^2\,.
\end{equation}
Using the ``dilaton" coordinate $\rho=2\pi r^2-c$ we obtain:

\begin{equation}\label{metrictheta0term}
    \begin{split}
    2\pi g_N -\frac{2}{f}(\theta_0^P)^2=& -\frac{\rho +2c -f^{\text{inst}}}{4(\rho+c)(\rho+f^{\text{inst}})}d\rho^2 +(\rho+c)(\widetilde{\eta}^2-g_{\overline{M}}) + 2\pi\sum_{\gamma}\Omega(\gamma)V_{\gamma}^{\text{inst}}|dZ_{\gamma}|^2 + \frac{1}{2\pi}Y_iM^{ij}\overline{Y}_j \\
    &-\frac{1}{2(\rho+f^{\text{inst}})}\Big( 2d\rho df^{\text{inst}}+(df^{\text{inst}})^2\Big)\,,
    \end{split}
\end{equation}
where we combined the $d\rho^2$ terms coming from $2\pi dr^2$ and $df^2$. The change to $\rho$ is done to compare more easily with $g_{\text{FS}}^c$.  \\

Finally, since $\theta_2^P$ and $\theta_3^P$ are the real and imaginary part of the same complex form, we get

\begin{equation}\label{theta23term}
    -\frac{2}{f}((\theta_2^P)^2+(\theta_3^P)^2)=-\frac{2}{f} \Big|Z^iW_i+Z^iW_i^{\text{inst}}+i\Big(\sum_{\gamma}\Omega(\gamma)2\pi A_{\gamma}^{\text{inst}}(V)dZ_{\gamma}\Big)\Big|^2\,.
\end{equation}

Combining the results (\ref{linebundleterm3}), (\ref{metrictheta0term}), (\ref{theta23term}); writing the coefficients in terms of $\rho$, the constant $c$ and instanton correction terms; and using the notation
\begin{equation}
    f_+^{\text{inst}}:=(f^{\text{inst}}+f_1^{\text{inst}})/2, \;\;\;\;\;\;\; f_-^{\text{inst}}:=(f^{\text{inst}}-f_1^{\text{inst}})/2\,,
\end{equation}
we obtain
\begin{equation}\label{coordQKmetric}
    \begin{split}
    -\frac{1}{f}\widetilde{g}_P=&\frac{\rho+c}{\rho+f^{\text{inst}}}g_{\overline{M}}-\frac{2\pi}{\rho+f^{\text{inst}}}\sum_{\gamma}\Omega(\gamma)V_{\gamma}^{\text{inst}}|dZ_{\gamma}|^2\\
    &+\frac{1}{2(\rho+f^{\text{inst}})^2}\Big(\frac{\rho +2c -f^{\text{inst}}}{2(\rho+c)}d\rho^2+2d\rho df^{\text{inst}}+(df^{\text{inst}})^2\Big)\\
    &+4\frac{\rho+c+f_-^{\text{inst}}}{(\rho+f^{\text{inst}})^2(\rho+2c-f_1^{\text{inst}})}\Big(\Theta +\frac{c}{2}\widetilde{\eta}+\eta_+^{\text{inst}}+\frac{f_+^{\text{inst}}-c}{\rho+c+f_-^{\text{inst}}}\eta_-^{\text{inst}}\Big)^2\\
    &-\frac{1}{2\pi(\rho+f^{\text{inst}})}(W_i+W_i^{\text{inst}})(N+N^{\text{inst}})^{ij}(\overline{W}_j+\overline{W}_j^{\text{inst}})\\
    &+\frac{2}{(\rho+f^{\text{inst}})^2} \Big|Z^iW_i+Z^iW_i^{\text{inst}}+i\Big(\sum_{\gamma}\Omega(\gamma)2\pi A_{\gamma}^{\text{inst}}(V)dZ_{\gamma}\Big)\Big|^2\\
    &
    +\frac{\rho+c+f_-^{\text{inst}}}{\rho+f^{\text{inst}}}\Big(\widetilde{\eta}-\frac{2}{\rho+c+f_-^{\text{inst}}}\eta_-^{\text{inst}}\Big)^2-\frac{\rho+c}{\rho+f^{\text{inst}}}\widetilde{\eta}^2\,.
    \end{split}
\end{equation}

Now we wish to restrict the previous expression to $\overline{N}$ and use the coordinates $\rho$, $X^i$ (for $i=1,2,...,n$), $\widetilde{\theta}_i$, $\theta^i$ and $\sigma$ previously defined before Theorem \ref{propcoordQK}. We also recall the normalized central charge $X_{\gamma}=Z_{\gamma}/z^0$. In particular, for $\gamma \in \text{Supp}(\gamma)$ with $\gamma=n_i(\gamma)\partial_{y_i}$ we have $X_{\gamma}=n_i(\gamma)X^i$ with $X^0=1$.\\

 With respect to these coordinates,
    \begin{equation}\label{z^0rho}
        |z^0|^2=r^2e^{\mathcal{K}}=\frac{\rho+c}{2\pi}e^{\mathcal{K}}\,, 
    \end{equation}
    where $\mathcal{K}$ is a K\"{a}hler potential for the PSK metric $g_{\overline{M}}$ given by $\mathcal{K}=-\log(N_{ij}X^i\overline{X}^j)$.\\ 
    
    Using (\ref{z^0rho}) we can write
    \begin{equation}
        \widetilde{\eta}=\frac{1}{2}d^c\log(r^2)=d(\text{Arg}(z^0)) - \frac{1}{2}d^c\mathcal{K}\,,
    \end{equation}
    so that $\widetilde{\eta}|_{\overline{N}}=-\frac{1}{2}d^c\mathcal{K}$.\\

Using (\ref{z^0rho}) and $|z^0|^2|_{\overline{N}}=(z^0)^2|_{\overline{N}}$, we find that for $\gamma \in \text{Supp}(\Omega)$,

\begin{equation}\label{usefulexp}
dZ_{\gamma}/z^0\Big|_{\overline{N}}=n_i(\gamma)\Big(dX^i +X^i\frac{dz^0}{z^0}\Big)\Big|_{\overline{N}}=dX_{\gamma}+X_{\gamma}\Big(\frac{d\rho}{2(\rho+c)}+\frac{d\mathcal{K}}{2}\Big)\,,
\end{equation}
so that
\begin{equation}
    \begin{split}
-&\frac{2\pi}{\rho+f^{\text{inst}}}\sum_{\gamma}\Omega(\gamma)V_{\gamma}^{\text{inst}}|dZ_{\gamma}|^2\Big|_{\overline{N}}=-e^{\mathcal{K}}\frac{\rho+c}{\rho+f^{\text{inst}}}\sum_{\gamma}\Omega(\gamma)V_{\gamma}^{\text{inst}}\Big|dX_{\gamma} + X_{\gamma}\Big(\frac{d\rho}{2(\rho+c)}+\frac{d\mathcal{K}}{2}\Big)\Big|^2\,.
\end{split}
\end{equation}

Furthermore, we can write

\begin{equation}
    \begin{split}
        &\frac{2}{(\rho+f^{\text{inst}})^2} \Big|Z^iW_i+Z^iW_i^{\text{inst}}+i\Big(\sum_{\gamma}\Omega(\gamma)2\pi A_{\gamma}^{\text{inst}}(V)dZ_{\gamma}\Big)\Big|^2\Big|_{\overline{N}}\\
        &=\frac{(\rho+c)e^{\mathcal{K}}}{\pi(\rho+f^{\text{inst}})^2}\Big|X^iW_i+X^iW_i^{\text{inst}}|_{\overline{N}}+2\pi i\Big(\sum_{\gamma}\Omega(\gamma) A_{\gamma}^{\text{inst}}(V)\Big(dX_{\gamma} + X_{\gamma}\Big(\frac{d\rho}{2(\rho+c)}+\frac{d\mathcal{K}}{2}\Big)\Big)\Big|^2\,.
    \end{split}
\end{equation}

Hence, putting all together we obtain the desired expression for $g_{\overline{N}}$ given in (\ref{coordQKmetric2}).\\

We also give the coordinate expressions for $df^{\text{inst}}|_{\overline{N}}$, $W_i^{\text{inst}}|_{\overline{N}}$ and $\eta^{\text{inst}}_{\pm}|_{\overline{N}}$. By Lemma \ref{ff1} we have $df=-4\pi \iota_V\omega_3$, so by (\ref{invKF}):

\begin{equation}
        df^{\text{inst}}=2\pi\Big(\sum_{\gamma}\Omega(\gamma)V^{\text{inst}}_{\gamma}(Z_{\gamma}d\overline{Z}_{\gamma}+\overline{Z}_{\gamma}dZ_{\gamma})-\sum_{\gamma}\frac{i\Omega(\gamma)}{2\pi^2}\Big(\sum_{n>0}e^{in\theta_{\gamma}}|Z_{\gamma}|K_1(2\pi n|Z_{\gamma}|)\Big)d\theta_{\gamma}\Big)\,.
    \end{equation}
On the other hand, using (\ref{z^0rho}) and (\ref{usefulexp})

\begin{equation}
    \begin{split}
    Z_{\gamma}d\overline{Z}_{\gamma}+\overline{Z}_{\gamma}dZ_{\gamma}\Big|_{\overline{N}}&=\frac{(\rho+c)}{2\pi}e^{\mathcal{K}}\Big(X_{\gamma}\Big(d\overline{X}_{\gamma} + \overline{X}_{\gamma}\Big( \frac{d\rho}{2(\rho+c)}+\frac{d\mathcal{K}}{2}\Big)\Big)+\overline{X}_{\gamma}\Big(dX_{\gamma} + X_{\gamma}\Big( \frac{d\rho}{2(\rho+c)}+\frac{d\mathcal{K}}{2}\Big)\Big)\\
    &=\frac{(\rho+c)}{\pi}e^{\mathcal{K}}\Big(\text{Re}(X_{\gamma}d\overline{X}_{\gamma}) + |X_{\gamma}|^2\Big( \frac{d\rho}{2(\rho+c)}+\frac{d\mathcal{K}}{2}\Big)\Big)\,,
    \end{split}
\end{equation} so that

\begin{equation}\label{dfinst}
    \begin{split}
        df^{\text{inst}}|_{\overline{N}}=& 2(\rho+c)e^{\mathcal{K}}\sum_{\gamma}\Omega(\gamma)V^{\text{inst}}_{\gamma}|_{\overline{N}}\Big(\text{Re}(X_{\gamma}d\overline{X}_{\gamma}) + |X_{\gamma}|^2\Big( \frac{d\rho}{2(\rho+c)}+\frac{d\mathcal{K}}{2}\Big)\Big)\\
        &-\sqrt{\frac{\rho+c}{2\pi}}e^{\mathcal{K}/2}\sum_{\gamma}\frac{i\Omega(\gamma)}{\pi}\Big(\sum_{n>0}e^{in\theta_{\gamma}}|X_{\gamma}|K_1\Big(2\pi n\sqrt{\frac{\rho+c}{2\pi}}e^{\mathcal{K}/2}|X_{\gamma}|\Big)\Big)d\theta_{\gamma}\,,
    \end{split}
\end{equation}
where

\begin{equation}
    V_{\gamma}^{\text{inst}}|_{\overline{N}}=\frac{1}{2\pi}\sum_{n>0}e^{in\theta_{\gamma}}K_0\Big(2\pi n\sqrt{\frac{\rho+c}{2\pi}}e^{\mathcal{K}/2}|X_{\gamma}|\Big)\Big)\,.
\end{equation}

On the other hand, using (\ref{z^0rho}) and (\ref{usefulexp}) again, we find that

\begin{equation}\label{restrictedforms}
        \begin{split}
           \Big(\frac{dZ_{\gamma}}{Z_{\gamma}}-\frac{d\overline{Z}_{\gamma}}{\overline{Z}_{\gamma}}\Big)\Big|_{\overline{N}}&=\Big(\frac{dX_{\gamma}}{X_{\gamma}}-\frac{d\overline{X}_{\gamma}}{\overline{X}_{\gamma}}\Big)\\
           \iota_V(|dZ_{\gamma}|^2)|_{\overline{N}}&=-\frac{(\rho+c)e^{\mathcal{K}}}{2\pi}\text{Im}(X_{\gamma}d\overline{X}_{\gamma})\,.
        \end{split}
    \end{equation}
Using the equalities (\ref{restrictedforms}) we can compute $W_i^{\text{inst}}|_{\overline{N}}$ and $\eta_{\pm}^{\text{inst}}|_{\overline{N}}$. For $W_i^{\text{inst}}|_{\overline{N}}$ we obtain
\begin{equation}\label{winst}
    W_i^{\text{inst}}|_{\overline{N}}=-\sum_{\gamma}\Omega(\gamma)n_i(\gamma)\Big[ \frac{1}{2}\sqrt{\frac{\rho+c}{2\pi}}e^{\mathcal{K}/2}\sum_{n>0}e^{in\theta_{\gamma}}|X_{\gamma}|K_1\Big(2\pi n\sqrt{\frac{\rho+c}{2\pi}}e^{\mathcal{K}/2}|X_{\gamma}|\Big)\Big)\Big(\frac{dX_{\gamma}}{X_{\gamma}}-\frac{d\overline{X}_{\gamma}}{\overline{X}_{\gamma}}\Big)+iV_{\gamma}^{\text{inst}}|_{\overline{N}}d\theta_{\gamma}\Big]\,,
\end{equation}
while from (\ref{etagammainst}) and (\ref{etainstcoord}) we obtain
\begin{equation}\label{eta1inst}
    \begin{split}
    \eta_{\gamma}^{\text{inst}}|_{\overline{N}}&=\frac{i}{8\pi^2}\sqrt{\frac{\rho+c}{2\pi}}e^{\mathcal{K}/2}\Big(\sum_{n>0} \frac{e^{in\theta_{\gamma}}}{n}|X_{\gamma}|K_1\Big(2\pi n\sqrt{\frac{\rho+c}{2\pi}}e^{\mathcal{K}/2} |X_{\gamma}|\Big)\Big)\Big(\frac{dX_{\gamma}}{X_{\gamma}}-\frac{d\overline{X}_{\gamma}}{\overline{X}_{\gamma}}\Big)\\
    \eta^{\text{inst}}|_{\overline{N}}=&  \sum_{\gamma}\Omega(\gamma)\Big(\eta_{\gamma}^{\text{inst}}|_{\overline{N}}+V_{\gamma}^{\text{inst}}|_{\overline{N}}\text{Im}(X_{\gamma}d\overline{X}_{\gamma})\Big)-\frac{M^{ij}|_{\overline{N}}}{4\pi^2}W_i^{\text{inst}}(V)|_{\overline{N}}\text{Re}(W_j+W_j^{\text{inst}}|_{\overline{N}})\,.
    \end{split}
\end{equation}
From the last equations we can find the expressions for $\eta_{\pm}^{\text{inst}}|_{\overline{N}}$ from (\ref{etapm}). In particular, we conclude that $W_i^{\text{inst}}|_{\overline{N}}$ and $\eta_{\pm}^{\text{inst}}|_{\overline{N}}$ have no $d\rho$ components. 
\end{proof}
\end{section}

\bibliography{References}

\begin{thebibliography}{ACDM15}

\bibitem[AB15]{HMmetric}
S.~Alexandrov and S.~Banerjee.
\newblock Hypermultiplet metric and {D}-instantons.
\newblock {\em J. High Energy Phys.}, 2015(2), 2015.

\bibitem[ACD02]{ACD}
D.~Alekseevsky, V.~Cort\'es, and C.~Devchand.
\newblock Special complex manifolds.
\newblock {\em J. Geom. Phys}, 42(1-2):85--105, 2002.

\bibitem[ACDM15]{QKPSK}
D.V. Alekseevsky, V.~Cortés, M.~Dyckmanns, and T.~Mohaupt.
\newblock Quaternionic {K}ähler metrics associated with special {K}ähler
  manifolds.
\newblock {\em J. Geom. Phys}, 92:271--287, 2015.

\bibitem[ACM13]{Conification}
D.V. Alekseevsky, V.~Cort\'es, and T.~Mohaupt.
\newblock Conification of {K}ähler and hyper-{K}ähler manifolds.
\newblock {\em Comm. Math. Phys.}, 324(2):637–655, 2013.

\bibitem[Ale07]{1loopcmap}
Sergei Alexandrov.
\newblock Quantum covariant c-map.
\newblock {\em Journal of High Energy Physics}, 2007(05):094–094, May 2007.

\bibitem[Ale13]{HMreview1}
S.~Alexandrov.
\newblock Twistor approach to string compactifications: A review.
\newblock {\em Phys. Rep.}, 522(1):1--57, 2013.

\bibitem[AMNP15]{AMNP}
S.~Alexandrov, G.~Moore, A.~Neitzke, and B.~Pioline.
\newblock $\mathbb{R}^3$-index for four dimensional $\mathcal{N}=2$ field
  theories.
\newblock {\em Phys. Rev. Lett}, 114(12), 2015.

\bibitem[AMPP15]{HMreview2}
S.~Alexandrov, J.~Manschot, D.~Persson, and B.~Pioline.
\newblock Quantum hypermultiplet moduli spaces in ${N}=2$ string vacua: a
  review.
\newblock {\em Proc.Symp.Pure Math.}, 90:181--212, 2015.

\bibitem[APP11]{WCHKQK}
S.~Alexandrov, D.~Persson, and B.~Pioline.
\newblock Wall-crossing, {R}ogers dilogarithm and the {QK}-{HK} correspondence.
\newblock {\em J. High Energy Phys.}, 2011(12), 2011.

\bibitem[APSV09]{Dinsttwist}
Sergei Alexandrov, Boris Pioline, Frank Saueressig, and Stefan Vandoren.
\newblock D-instantons and twistors.
\newblock {\em Journal of High Energy Physics}, 2009(03):044–044, Mar 2009.

\bibitem[APSV10]{linearpertQK}
Sergei Alexandrov, Boris Pioline, Frank Saueressig, and Stefan Vandoren.
\newblock Linear perturbations of quaternionic metrics.
\newblock {\em Communications in Mathematical Physics}, 296(2):353–403, Feb
  2010.

\bibitem[Bri19]{VarBPS}
T.~Bridgeland.
\newblock Riemann-{H}ilbert problems from {D}onaldson-{T}homas theory.
\newblock {\em Invent. Math.}, 216:69–124, 2019.

\bibitem[CDS17]{CDS}
V.~Cort\'es, M.~Dyckmanns, and S.~Suhr.
\newblock Completeness of projective {K}ähler and quaternionic {K}ähler
  manifolds.
\newblock {\em Springer INdAM Series Special Metrics and Group Actions in
  Geometry}, 23, 2017.

\bibitem[CFG89]{TypeIIgeometry}
S.~Cecotti, S.~Ferrara, and L.~Girardello.
\newblock Geometry of type {II} superstrings and the moduli of superconformal
  field theories.
\newblock {\em Internat. J. Modern Phys. A}, 4(10), 1989.

\bibitem[CM09]{CM}
V.~Cort\'es and T.~Mohaupt.
\newblock Special {G}eometry of {E}uclidean {S}upersymmetry {III}: the local
  r-map, instantons and black holes.
\newblock {\em J. High Energy Phys.}, 2009(07), 2009.

\bibitem[Cor98]{C}
V.~Cort\'es.
\newblock On hyper {K}ähler manifolds associated to {L}agrangian {K}ähler
  submanifolds of {$T^*\mathbb{C}^n$}.
\newblock {\em Trans.\ Amer.\ Math.\ Soc.}, 350(8), 1998.

\bibitem[CRT21]{CRT}
V.~Cortés, M.~Röser, and D.~Thung.
\newblock Complete quaternionic k\"ahler manifolds with finite volume ends.
\newblock {\em arXiv:2105.00727 [math.DG]}, 2021.

\bibitem[Fre99]{SK}
D.~Freed.
\newblock Special {K}ähler manifolds.
\newblock {\em Comm. Math. Phys}, 203(1):31–52, 1999.

\bibitem[FS90]{FSmetric}
S.~Ferrara and S.~Sabharwal.
\newblock Quaternionic manifolds for type {II} superstring vacua of
  {C}alabi-{Y}au spaces.
\newblock {\em Nucl. Phys. B}, 332(2):317--332, 1990.

\bibitem[GMN10]{GMN}
D.~Gaiotto, G.~Moore, and A.~Neitzke.
\newblock Four-dimensional wall-crossing via three-dimensional field theory.
\newblock {\em Comm. Math. Phys.}, 299(1):163–224, 2010.

\bibitem[GW00]{LCS}
M.~Gross and P.~M.~H. Wilson.
\newblock Large complex structure limits of {K}3 surfaces.
\newblock {\em J. Differential Geom.}, 55(3), 2000.

\bibitem[Hay08]{HKQK}
A.~Haydys.
\newblock Hyperkähler and quaternionic {K}ähler manifolds with an ${S}^1$
  symmetry.
\newblock {\em J. Geom. Phys}, 58(3):293--306, 2008.

\bibitem[Hit87]{SDeq}
N.~Hitchin.
\newblock The self-duality equations on a {R}iemann surface.
\newblock {\em Proc. Lond. Math. Soc.}, 55(3):59--126, 1987.

\bibitem[Hit13]{HitHKQK}
N.~Hitchin.
\newblock On the hyperkähler/ quaternion {K}ähler correspondence.
\newblock {\em Comm. Math. Phys}, 324(1):77–106, 2013.

\bibitem[KS08]{KS}
M.~Kontsevich and Y.~Soibelman.
\newblock Stability structures, motivic {D}onaldson-{T}homas invariants and
  cluster transformations.
\newblock {\em arXiv:0811.2435 [math.AG]}, 2008.

\bibitem[Nei14]{NewCHK}
A.~Neitzke.
\newblock Notes on new constructions of hyperkähler metrics.
\newblock {\em Lect. Notes Unione Mat. Ital}, 15, 2014.

\bibitem[RLSV06]{RLSV}
D.~Robles-Llana, F.~Saueressig, and S.~Vandoren.
\newblock String loop corrected hypermultiplet moduli spaces.
\newblock {\em J. High Energy Phys}, 81(3), 2006.

\end{thebibliography}
\bibliographystyle{alpha}
\end{document}